\documentclass[11pt,reqno,english]{amsart}

\usepackage[letterpaper]{geometry}
\usepackage{todonotes}
\usepackage{latexsym,amssymb,amsmath,amsfonts,amsthm,xcolor,graphicx, longtable}
\usepackage[utf8]{inputenc}
\usepackage{textgreek}
\allowdisplaybreaks

\usepackage{amsmath}
\usepackage{graphicx}
\usepackage{float}
\usepackage{color}
\usepackage{pifont, hyperref}
\usepackage{fancyhdr,listings,amsfonts,amsmath,color,amsthm}
\hypersetup{colorlinks=true,linkcolor=blue}
\usepackage{latexsym,amssymb,amsmath,amsfonts,amsthm,xcolor,graphicx}
\usepackage[utf8]{inputenc}
\usepackage{amsmath}
\usepackage{graphicx,enumitem}
\usepackage{float}
\usepackage{color}
\usepackage{pifont}
\usepackage{amssymb}
\usepackage{multirow}
\usepackage{diagbox}
\usepackage{cancel}
\usepackage{bbm}
\usepackage{xcolor}
\usepackage[ruled,vlined]{algorithm2e}
\usepackage{algpseudocode}
\usepackage[capitalise]{cleveref}
\crefname{algocf}{alg.}{algs.}
\Crefname{algocf}{Algorithm}{Algorithms}

\newcommand{\cpoincare}{C_{\textsc{P}}}
\newcommand{\chpoincare}{C_{h,\textsc{P}}}

\newcommand{\clsi}{C_{\textsc{LSI}}}
\newcommand{\chlsi}{C_{h,\textsc{LSI}}}

\newcommand{\cmlsi}{C_{\textsc{m-LSI}}}
\newcommand{\chmlsi}{C_{h,\textsc{m-LSI}}}

\newcommand{\ctail}{C_{\textsc{tail}}}
\newcommand{\chtail}{C_{h,\textsc{tail}}}

\usepackage{mathtools}

\newcommand{\lv}{\left\lVert}
\newcommand{\rv}{\right\rVert}

\newcommand{\mb}{\mathbb}
\newcommand{\mc}{\mathcal}

\newcommand{\p}{\pi}

\newtheorem{theorem}{Theorem}
\newtheorem{lemma}{Lemma}
\newtheorem{corollary}{Corollary}
\newtheorem*{definition}{Definition}
\newtheorem{proposition}{Proposition}
\theoremstyle{definition}
\newtheorem{assump}{Assumption}
\newenvironment{myassump}[2][]
  {\renewcommand\theassump{#2}\begin{assump}[#1]}
  {\end{assump}}
\newtheorem{remark}{Remark}
\usepackage{mathtools}

\newcommand*\Laplace{\mathop{}\!\mathbin\bigtriangleup}

\title{Heavy-tailed Sampling via Transformed Unadjusted Langevin Algorithm}

\author{Ye He}
\address{Department of Mathematics, University of California, Davis.} \email{leohe@ucdavis.edu}

\author{Krishnakumar Balasubramanian}
\address{Department of Statistics, University of California, Davis.}
\email{kbala@ucdavis.edu}

\author{Murat A. Erdogdu}
\address{Department of Computer Science and Department of Statistical Sciences, University of Toronto}
\email{erdogdu@cs.toronto.edu}

\begin{document}

\maketitle
\begin{abstract}
We analyze the oracle complexity of sampling from polynomially decaying heavy-tailed target densities based on running the Unadjusted Langevin Algorithm on certain transformed versions of the target density. The specific class of closed-form transformation maps that we construct are shown to be diffeomorphisms, and are particularly suited for developing efficient diffusion-based samplers. We characterize the precise class of heavy-tailed densities for which polynomial-order oracle complexities (in dimension and inverse target accuracy) could be obtained, and provide illustrative examples. We highlight the relationship between our assumptions and functional inequalities (super and weak Poincar\'e inequalities) based on non-local Dirichlet forms defined via fractional Laplacian operators, used to characterize the heavy-tailed equilibrium densities of certain stable-driven stochastic differential equations. 
\end{abstract}

\section{Introduction}  \label{sec:TULA}
 Given a potential function $f:\mathbb{R}^d \to \mathbb{R}$, we consider the problem of sampling from the density
\begin{align}\label{eq:target}
\pi(x) \coloneqq Z^{-1}e^{-f(x)},
\end{align}
where $Z\coloneqq \smallint e^{-f(x)}dx$ is an (unknown) normalization constant. A general strategy to sample from densities of the form in~\eqref{eq:target} is to discretize a diffusion equation that has $\pi$ as its stationary density. In particular, the over-damped Langevin diffusion described by the Stochastic Differential Equation (SDE)
\begin{align}\label{LDy}
    dX_t=-\nabla f(X_t) dt +\sqrt{2} d W_t,
\end{align}
where $W_t$ is a $d$-dimensional Brownian motion, has attracted considerable attention in the past decade. Under mild regularity conditions, the diffusion in~\eqref{LDy} has $\pi$ as its stationary density, which provides the motivation for the above approach. In this work, we are interested in the case when the density $\pi$ has heavy-tails (for example, tails that are polynomially decaying). Sampling from such heavy-tailed densities arise in various applications including robust statistics~\cite{kotz2004multivariate, jarner2007convergence, kamatani2018efficient}, multiple comparison procedures~\cite{genz2004approximations, genz2009computation} and statistical machine learning~\cite{nguyen2019non, simsekli2020fractional}. 

When the target density $\pi$ is heavy-tailed, the solution to~\eqref{eq:target} is not exponentially ergodic, that is, the solution does not converge to the stationary density rapidly. Indeed~\cite[Theorem 2.4]{roberts1996exponential} shows that if $|\nabla f(x)| \to 0$ when $|x| \to \infty$, then the solution to~\eqref{LDy} is \emph{not} exponentially ergodic. In the other direction, standard results in the literature, for example~\cite{wang2006functional, bakry2014analysis} show that the solution to~\eqref{LDy} being exponentially ergodic is equivalent to the density $\pi$ satisfying the Poincar\'e inequality, which requires $\pi$ to have exponentially decaying tails. Furthermore,~\cite[Chapter 4]{wang2006functional} shows that when $\pi$ has  polynomially decaying tails, the convergence is only sub-exponential or polynomial. 

Turning to time-discretizations of \eqref{LDy}, the Euler discretization or the Unadjusted Langevin Algorithm (ULA) is given by 
\begin{align}\label{ULA}
    x_{n+1}=x_n-\gamma \nabla f(x_n) +\sqrt{2\gamma} u_{n+1},  
\end{align}
where $(u_n)$ is a sequence of independent and identically distributed $d$-dimensional standard Gaussian vectors and $\gamma>0$ is a user-defined step size parameter. Over the past decade, non-asymptotic oracle complexity analysis of ULA (and other related discretizations) have been studied intensively. We refer to~\cite{dalalyan2017theoretical, durmus2017nonasymptotic, dalalyan2019user,durmus2019analysis, lee2020logsmooth, dwivedi2019log, shen2019randomized, he2020ergodicity, chen2020fast,chewi2021optimal, wu2021minimax} for the case when the potential $f$ is strongly convex,~\cite{durmus2019analysis, dalalyan2019bounding, chen2020fast, lehec2021langevin} when it is convex, and~\cite{cheng2018sharp,ma2019sampling, majka2020nonasymptotic} when it is non-convex. We also highlight the works of~\cite{vempala2019rapid},~\cite{erdogdu2020convergence},~\cite{nguyen2021unadjustedb} and~\cite{chewi2021analysis} which analyzed ULA when $\pi$ satisfies certain functional inequalities. Specifically,~\cite{vempala2019rapid} showed that when $\pi$ satisfies a Logarithmic Sobolev Inequality (LSI) and has Lipschitz-smooth gradients, ULA with a number of iterations of order $\Tilde{O}(1/\epsilon)$ generates a sample which is $\epsilon$-close to $\pi$ with respect to KL-divergence. A necessary condition for $\pi$ to satisfy the LSI condition is that it should have sub-Gaussian tails. Furthermore,~\cite{erdogdu2020convergence} considered densities that satisfy a modified LSI (m-LSI) inequality and showed that the number of iterations becomes of order $\Tilde{O}(1/\epsilon^c)$, for some $c\geq 1$ (which depends on certain smoothness conditions). A typical example of a density that satisfies a m-LSI condition but not the LSI condition is $\pi(x)\propto \exp(-|x|)$. Thus, the result in~\cite{erdogdu2020convergence} could also be viewed as an oracle complexity result for ULA when sampling from sub-exponential densities. Recently~\cite{nguyen2021unadjustedb} relaxed the conditions required in~\cite{erdogdu2020convergence} and provided similar results under the assumption that the target density satisfies a Poincar\'e inequality and dissipativity at the same time. Furthermore,~\cite{chewi2021analysis} also presented an analysis of ULA under the so-called Lata{\l}a-Oleszkiewicz~\cite{latala2000between} inequality, that interpolates between the LSI and Poincar\'{e} inequality for the stronger R\'enyi metric and removes the dissipativity assumptions required in~\cite{erdogdu2020convergence,nguyen2021unadjustedb}. It is worth pointing out here that the proofs of~\cite{erdogdu2020convergence} and~\cite{chewi2021analysis} are based on certain transformations of the target densities.

The above results, however, are not applicable to sampling from polynomially decaying heavy-tailed densities like the multivariate $t$-distribution, whose density is of the form $\pi(x)\propto (1+|x|^2)^{-\frac{d+\kappa}{2}}$, where $\kappa>0$ is the degrees-of-freedom parameter. Recently, some attempts have been made to sample from such heavy-tailed densities by considering stable-driven SDEs of the form 
\begin{align}\label{eq:stablesde}
        dX_t=b(X_t)dt +\sqrt{2} dZ_t
\end{align}
where $b$ is the drift term defined based on the Riesz potential, and $Z_t$ is an $\vartheta$-stable process with $\vartheta \in (1,2)$~\cite{simsekli2020fractional,csimcsekli2017fractional, huang2021approximation}. Specifically~\cite{huang2021approximation} established exponential ergodicity of the solution of~\eqref{eq:stablesde}, under conditions that allow for much heavier tails than Brownian-driven SDEs. The eventual hope is that discretizations of~\eqref{eq:stablesde} might lead to algorithms with provable non-asymptotic oracle complexity rates. However, it appears to be non-trivial to analyze discretization of~\eqref{eq:stablesde}, especially if we are interested in tight  non-asymptotic results, due to the difficulties in dealing with the non-smoothness of drift term $b$. 

In this work, we take an alternate approach for heavy-tailed sampling using ULA on a transformed version of the target density. Such an approach was used by \cite{johnson2012variable} in the context of Metropolis Random Walk algorithm, which serves as our motivation. The key idea in this approach is to construct smooth invertible maps (also called diffeomorphisms) $h:\mathbb{R}^d\to\mathbb{R}^d$ that transform the heavy-tailed density $\pi$ to an appropriately light-tailed density $\pi_h$. Given such a map, one could first sample from the light-tailed density $\pi_h$ and subsequently obtain samples from the heavy-tailed density $\pi$ using the inverse map of $h$. It is also worth highlighting that~\cite{deligiannidis2019exponential,durmus2020geometric} and~\cite{bierkens2019ergodicity} used the transformation approach for proving asymptotic exponential ergodicity of bouncy particle and zig-zag samplers respectively, in the heavy-tailed setting.

There are several issues to overcome when using the above strategy in the context of ULA. First, note that the constructed map $h$ has to convert the heavy-tailed density $\pi$ to a light-tailed density $\pi_h$. In this process, however, the bulk of the density $\pi_h$ might become non-smooth, if the map is not constructed carefully. This non-smoothness could subsequently hinder the usage of ULA algorithm to sample from $\pi_h$. Second, the constants involved (for example, the LSI or m-LSI constant) in the light-tailed density $\pi_h$ might start to depend exponentially on the dimension after transformation. This again hinders the efficiency of the ULA when sampling from $\pi_h$. Furthermore, the transformation map needs to be efficiently computable. In this work, we propose a family of carefully constructed transformations that overcome the above issues and present non-asymptotic results for sampling from a class of heavy-tailed densities.

\subsection{Other Related Approaches} We briefly also highlight two other potential approaches for sampling from heavy-tailed densities. Instead of considering isotropic Langevin diffusions of the form in~\eqref{LDy}, one could also consider general classes of It\^{o} diffusions still driven by Brownian motion. However, it is not clear how to specify the drift and diffusion terms of the It\^{o} diffusion for a given target density. In this context, the proposed transformations could also be interpreted as providing a principled way of constructing the drift and diffusion terms; we elaborate more on this in~\cref{sec:transformationvsIto}. Relatively few works exist on analyzing discretization of general It\^{o} diffusions; see, for example, \cite{erdogdu2018global, li2019stochastic}. Furthermore, the conditions on the allowed class of  It\^{o} diffusions and their corresponding invariant measures in~\cite{erdogdu2018global, li2019stochastic} need to imply contraction in Wasserstein metric, which may not be satisfied in the polynomially decaying heavy-tailed regimes we focus on in this work.

Yet another approach is to learn a transformation that maps an easily samplable density to the heavy-tailed target density~\cite{erbar2014gradient, spantini2018inference, parno2016multiscale, cui2021deep, teshima2020coupling, koehler2021representational, papamakarios2021normalizing}. In the literature, popular approaches to learn such transformations are based on parametrization with neural networks or tensor-trains. Compared to this approach, we can provide \emph{closed form representations} for the transformation and its inverse. Furthermore, to our knowledge non-asymptotic oracle complexity results for learned transformations are lacking even in the light-tailed case.   

\subsection{Organization} The rest of the paper is organized as follows. In ~\cref{sec:prelim} we introduce the notation and preliminary background material used in the rest of the paper. In Section~\ref{sec:transformations}, we introduce our transformation map, highlight key properties and present the Transformed Unadjusted Langevin Algorithm (TULA) algorithm. We also discuss a warm-up example regarding exponentially tailed densities, and provide an interpretation of the transformed diffusion as a special case of It\'o diffusions. In Section~\ref{sec:THM}, we present non-asymptotic oracle complexity results for TULA under various assumptions on the potential function that characterize the level of heavy-tails allowed. In Section~\ref{sec:htfunctionalineq}, we discuss the relationship between our assumptions on the heavy-tails used in Section~\ref{sec:THM} and non-local Dirichlet form based functional inequalities (arising in the equilibrium analysis of stable-driven diffusions). Illustrative examples are provided in Section~\ref{sec:Examples}.

\section{Preliminaries}\label{sec:prelim}

\noindent\textbf{Notation:} For a vector $a \in \mathbb{R}^d$, we represent the Euclidean norm by $|a|$. For a mapping $h:\mathbb{R}^d \to \mathbb{R}^d$, we denote the Jacobian matrix by $\nabla h \in \mathbb{R}^{d \times d}$. In the case when $h:\mathbb{R}^d \to\mathbb{R}$, $\nabla h \in\mathbb{R}^d$ denotes the gradient vector and $\Laplace  h = \nabla \cdot \nabla h $ denotes the Laplacian. For a function $h:\mathbb{R} \to \mathbb{R}$, we simply denote its first, second and third order derivatives by $h'$, $h''$ and $h'''$ respectively. For a matrix $A$, we denote its determinant and operator norm by $\det(A)$ and $\|A\|$ respectively. For two symmetric matrices $A,B$, the relation $A \preceq B$ refers to the fact that $B-A$ is positive semi-definite. The class of function $\mathcal{C}^k(\Omega)$ refers to those functions that have $k$-times continuously differentiable derivatives on the domain $\Omega$. For a function $\phi$, $\|\phi\|_\infty$ refers to the $\sup$-norm.


We also require the following definitions used in the rest of the paper. Let $\nu$ and $\mu$ be two probability densities with full support on $\mb{R}^d$. Then, for a convex function $\Phi:\mathbb{R} \to \mathbb{R}$ such that $\Phi(1) =0$, the \textit{$\Phi$-divergence} of $\nu$ from $\mu$ is defined as 
\begin{align*}
    D_\Phi(\nu|\mu)\coloneqq \int_{\mb{R}^d} \Phi\left(\frac{\nu(x)}{\mu(x)}\right)\mu(x)dx.
\end{align*}
When the function is given by $\Phi(t) = t \log(t)$, we obtain the \textit{Kullback-Leibler (KL) divergence} of $\nu$ with respect to $\mu$, given by
        \begin{align*}
          H_\mu(\nu)\coloneqq\int_{\mb{R}^d} \log \frac{\nu(x)}{\mu(x)} \nu(x) dx.
        \end{align*}
Our complexity results later will be provided in terms of KL-divergence. The \textit{Relative Fisher Information} of $\nu$ with respect to $\mu$ is given by
    \begin{align*}
      I_\mu(\nu)\coloneqq \int_{\mb{R}^d} \left|\nabla \log \frac{\nu(x)}{\mu(x)}\right|^2\nu(x) dx.
    \end{align*}
The R\'enyi divergence of order $q>1$ is defined as
\begin{align*}
R_{q}(\nu|\mu) = \frac{1}{q-1} \log \left(\int_{\mathbb{R}^d} \left(\frac{\nu(x)}{\mu(x)}\right)^q \mu(x) dx \right).
\end{align*}
Note that when $q\to1_+$, we have $R_q(\nu|\mu)$ approaching $H_\mu(\nu)$.

We now introduce additional technical details required for discussing functional inequalities; rigorous expositions could be found in~\cite{wang2006functional, bakry2014analysis}. Let $(\Omega, \mathcal{F}, \mu)$ be a probability space and let $\mathcal{L}$ denote a linear operator (infinitesimal generator) that is self-adjoint with domain $D(\mathcal{L})$ which generates a Markov semi-group $P_t$ on $L^2(\mu)$. The carr\'e de champ operator associated to the infinitesimal generator $\mathcal{L}$ is given by the bilinear map $\Gamma(\phi_1,\phi_2) = 1/2\left[ \mathcal{L}(\phi_1 \phi_2) - \phi_1 \mathcal{L} \phi_2  - \phi_2 \mathcal{L} \phi_1 \right]$, for all $\phi_1, \phi_2$ defined in a subspace of $D(\mathcal{L})$ which is an algebra. We call the collection of the measure $\mu$ on a state space $(\mb{R}^d, \mc{B}(\mb{R}^d))$ and a carr\'{e} de champ operator $\Gamma$ a Markov triple, denoted as $(\mb{R}^d,\mu,\Gamma)$. It is well-known that the Dirichlet form associated with a Markov semi-group $P_t$ is then given by  $\mathcal{E}(\phi_1,\phi_2) = \int \Gamma(\phi_1,\phi_2) d\mu$. By a standard integration-by-parts argument, we also have that $\mathcal{E}(\phi_1,\phi_2) = - \int \phi_1 \mathcal{L} \phi_2 $. We use the convention $\mathcal{E}(\phi)$ to denote $\mathcal{E}(\phi,\phi)$. The Dirichlet domain $\mc{D}(\mc{E})$ is defined as $\mc{D}(\mc{E})\coloneqq\left\{ \phi\in L^2(\mu): \mc{E}(\phi)<\infty \right\}$.

In the case of Brownian driven diffusions as in~\eqref{LDy}, the generator $\mathcal{L}$ is defined based on the Laplacian operator $\Laplace$, which is a local operator.  Correspondingly, the Dirichlet form is given by $\mathcal{E}(\phi) = \int |\nabla \phi(x)|^2 \mu(x) dx$, for all $\phi \in \mc{D}(\mc{E}) $. Based on this, we will introduce functional inequalities below. A probability density $\mu$ is said to satisfy \textit{Poincar\'e Inequality (PI)} with constant $\cpoincare$, denoted as $\mu \sim P(\cpoincare)$ if for all functions $\phi\in \mc{D}(\mc{E})$, we have 
\begin{align}\tag{PI}\label{eq:PI}
 \text{Var}_{\mu}(\phi):=  \int_{\mb{R}^d} \left( \phi^2(x) - \left( \int_{\mb{R}^d} \phi(x) \mu(x) dx\right) \right)^2 \mu(x) dx   \le \cpoincare \int_{\mb{R}^d} |\nabla \phi(x)|^2 \mu(x) dx = \cpoincare \mathcal{E}(\phi) .
\end{align}
Similarly, a probability density $\mu$ satisfies a \textit{Logarithmic Sobolev inequality (LSI)} with constant $\clsi$ denoted as $\mu\sim LS(\clsi)$ if for all functions $\phi\in \mc{D}(\mc{E})$, we have
    \begin{align}\tag{LSI}\label{eq:LSI}
   \text{Ent}_\mu(f):=     \int_{\mb{R}^d} f^2(x) \log\left(\frac{f^2(x)}{\int_{\mb{R}^d}f^2(x) \mu(x) dx} \right)\mu(x) dx   \le 2 \clsi \int_{\mb{R}^d} |\nabla f(x)|^2 \mu(x) dx = 2\clsi \mc{E}(\phi) .
    \end{align}
An equivalent form of LSI is that  for all probability densities $\rho(x)$, we have
    \begin{align*}
        H_{\mu}(\rho) \le \frac{\clsi}{2} I_{\mu}(\rho).
    \end{align*}
We refer the reader to~\cite[Chapter 5]{bakry2014analysis} for the derivation of the equivalence. A probability density $\mu(x)$ satisfies a \textit{modified Log-Sobolev Inequality (m-LSI)} if for all probability measure $\rho(x)$ and all $s\geq 2$, there is $\delta\in[0,1/2)$ (depending on $s$) such that
        \begin{align}\tag{m-LSI}\label{eq:mLSI}
            H_\mu(\rho)\le \cmlsi I_{\mu}(\rho)^{1-\delta}M_s(\rho+\mu)^{\delta}.
        \end{align}
where $M_s(\rho)=\int_{\mb{R}^d} (1+|x|^2)^{s/2}\rho(x)dx$. This version of m-LSI was introduced by~\cite{erdogdu2020convergence} (also see \cite{chewi2021analysis}), motivated by a related definition from~\cite{toscani2000trend}. It is important to notice that the above version of m-LSI does not contain the Poincar\'e inequality as a special case, i.e., there exists densities that satisfy the above m-LSI inequality but not Poincar\'e inequality and vice versa. There exists other modifications to the LSI including the Beckner or Nash inequality \cite[Chapter 7]{bakry2014analysis} and the Lata{\l}a-Oleszkiewicz~\cite{latala2000between} refinement to it, that interpolate between the LSI and Poincar\'e inequalities.

The above discussion is focused on Brownian driven SDEs. It turns out that the above class of functional inequalities are suitable for characterizing light-tailed densities (i.e., tails that decay exponentially fast). In the case of $\vartheta$-stable driven diffusions as in~\eqref{eq:stablesde}, the generator is defined based on the non-local fractional Laplacian operator $(-\Laplace)^{-\frac{\vartheta}{2}}$; see, for example,~\cite{kwasnicki2017ten}. Correspondingly, in Section~\ref{sec:htfunctionalineq}, we present more general functional inequalities based on non-local Dirichlet forms that are suitable for characterizing heavy-tailed densities and discuss the connection between our assumptions and such functional inequalities.

\section{The Transformed Unadjusted Langevin Algorithm} \label{sec:transformations}
\subsection{Transformation Map}\label{Trf}
We start this section by stating the following important property satisfied by smooth invertible transformation maps $h:\mathbb{R}^d \to \mathbb{R}^d$. 
\begin{definition}[Transformed density functions] For a probability density $\mu(x)$ with full support in $\mb{R}^d$, its transformed density function under a smooth invertible transformation map (or a diffeomorphism) $h$ is given by $\mu_h(x)=\mu(h(x)) \det (\nabla h(x))$ for all $x\in \mb{R}^d$.
\end{definition}
If a random vector $X$ has density $\mu$, then we denote the density of the random vector $Y=h^{-1}(X)$, denoted as $\mu_h$, as the transformed density of $\mu$.
Note that in particular if $X$ admits density $\pi$ of the form in~\eqref{eq:target}, then $Y=h^{-1}(X)$ is distributed with density 
  \begin{align}\label{Tdens}
        \pi_h(y)= Z^{-1}e^{-f_h(y)}\quad \text{with}\quad f_h(y)=f(h(y))-\log \det (\nabla h(y)),
  \end{align}
being referred to as the transformed potential. In what follows, we assume that the potential function is isotropic. We emphasize that this assumptions is made for the sake of technical convenience -- it is possible to relax this assumption to certain mild regularity conditions on the density, at the expense of having a more cluttered exposition. 
\begin{myassump}{A0}
\label{ass:A0}
The initial potential function $f$ is isotropic, i.e $f(x)=f(|x|)$ and $f:\mb{R}\to \mb{R}$ is twice continuously differentiable.
\end{myassump}
\noindent Since $f$ is isotropic under assumption \cref{ass:A0}, we may consider $f$ to be a function defined on $\mb{R}_+$ as well. In the later context, we use $f(|x|)$ when we consider $f$ defined on $\mb{R}_+$ and we use $f(x)$ when it is defined on $\mb{R}^d$. Similarly, when we use $f'(|x|), f''(|x|)$ and so on, to represent the derivatives, we consider $f$ to be a function defined on $\mb{R}_+$.

 We now describe the construction of our specific transformation map. Our proposal is motivated by the work of~\cite{johnson2012variable}, who constructed similar maps to show exponential ergodicity of the Metropolis Random Walk (MRW) algorithm. It turns out that a direct application of their construction to analyze Langevin diffusions and their discretization, leads to worse dimensionality dependencies in the non-asymptotic oracle complexities. Indeed, this is expected as~\cite{johnson2012variable} predominantly focused on establishing asymptotic results. In order to proceed, we first define functions $g:\mathbb{R} \to \mathbb{R}$ which correspond to the first part of the transformation map construction. Specifically,  $g$ is defined based on initial function $g_{in}$ as 
\begin{equation}\label{G1}
    g(r)=\left\{
    \begin{aligned}
        &g_{in}(r),   &r< b^{-\frac{1}{\beta}},\\
        &e^{br^\beta} &r\ge b^{-\frac{1}{\beta}}.
    \end{aligned}
    \right.
\end{equation}
where $\beta\in(1,2]$. The initial function $g_{in}:[0,b^{-\frac{1}{\beta}})\to [0,e)$ satisfies the following assumption.
\begin{myassump}{G1}\label{ass:G0} The initial function $g_{in}:[0,b^{-\frac{1}{\beta}})\to [0,e)$ is onto, monotone increasing and twice continuously differentiable. Furthermore, it satisfies, 
 \begin{equation*}
     \begin{split}
     g_{in}(0)&=0\\
     \lim_{r\to b^{-\frac{1}{\beta}}_{-}} g_{in}(r)&=e,\qquad \\
     \lim_{r\to b^{-\frac{1}{\beta}}_{-}} g'_{in}(r)&=\beta b^{\frac{1}{\beta}} e,\\
     \lim_{r\to b^{-\frac{1}{\beta}}_{-}} g''_{in}(r)&=(2\beta^2-\beta)b^{\frac{2}{\beta}}e   ,\\
     \lim_{r\to b^{-\frac{1}{\beta}}_{-}} g'''_{in}(r)&=(5\beta^3-6\beta^2+2\beta)b^{\frac{3}{\beta}}e,
     \end{split}
 \qquad\qquad
     \begin{split}
     \lim_{r\to 0_+} \left|\frac{f'(g_{in}(r))g'_{in}(r)}{r}\right|&<\infty, \\
     \lim_{r\to 0_+} \left|\frac{\frac{d}{dr}\log g_{in}'(r)}{r}\right|&<\infty,\\
     \lim_{r\to 0_+} \left|\frac{\frac{d}{dr}\log \frac{g_{in}(r)}{r}}{r}\right|&<\infty,\\
     \lim_{r\to 0_+} \left|\frac{d^2}{dr^2}\log \frac{g_{in}(r)}{r}\right|&<\infty,\\
     \lim_{r\to 0_+} \left|\frac{d^2}{dr^2}\log g_{in}'(r) \right|&<\infty.
     \end{split}
\end{equation*}
\end{myassump}
 We now show that if $g_{in}$ satisfies \cref{ass:G0}, then $g$ is three times continuously differentiable and invertible on $\mb{R}$.
\begin{lemma}\label{PropertyG} For the function $g$ defined in \eqref{G1}, if $g_{in}$ satisfies \cref{ass:G0}, then we have
\begin{itemize}
    \item [(1)] $g\in \mc{C}^3((0,\infty))$,
    \item [(2)] $g$ is onto, strictly monotonically increasing, and hence invertible.
\end{itemize}
\end{lemma}
The proof of \cref{PropertyG} is provided in~\cref{AnalyTran}. We now show that under~\cref{ass:G0}, the $\Phi$-divergence is preserved after transformation. This property is important to eventually provide our convergence results for sampling.

\begin{proposition}\label{Phithm} Let $h:\mathbb{R}^d\to\mathbb{R}^d$ be a transformation map satisfying~\cref{ass:G0}. For any two probability densities $\nu$ and $\mu$ with full support on $\mb{R}^d$, let $\nu_h$ and $\mu_h$ be the two transformed densities under the map $h$. Then the $\Phi$-divergence is preserved after transformation, i.e., we have
\begin{align}\label{phiequ}
   D_\Phi(\nu|\mu) = D_\Phi(\nu_h|\mu_h).
\end{align}
\end{proposition}
\begin{proof}[Proof:] We start from the right side of \eqref{phiequ}:
\begin{align*}
    D_\Phi(\nu_h|\mu_h)&=\int_{\mb{R}^d} \Phi\left(\frac{\nu_h(y)}{\mu_h(y)}\right) \mu_h(y)dy=\int_{\mb{R}^d} \Phi\left(\frac{\nu(h(y)) \det (\nabla h(y))}{\mu(h(y)) \det (\nabla h(y))}\right) \mu(h(y)) \det (\nabla h(y)) dy \\
    &=\int_{\mb{R}^d} \Phi\left(\frac{\nu(x)}{\mu(x)}\right) \mu(x) dx=D_\Phi(\nu|\mu).
\end{align*}
The second identity follows by the change of variable $x=h(y)$ and noting $\det (\nabla h(y))>0$ under~\cref{ass:G0}.
\end{proof}

With the properties of $g$ introduced in \cref{PropertyG}, we can then further define the isotropic transformations $h:\mb{R}^d\to \mb{R}^d$:
\begin{equation}\label{Hmap}
    h (x)=\left\{
    \begin{aligned}
        & \frac{g(|x|)x}{|x|}\ \ \ \ x\neq 0,\\
        & 0 \ \ \ \ \ \ \ \ \ \ \ \ x=0.
    \end{aligned}
    \right.
\end{equation}
We call the map $y\mapsto x= h(y)$ to be the transformation map, which is isotropic. Furthermore, $h$ is also three times continuously differentiable and invertible on $\mb{R}^d$ and its inverse is 
 \begin{equation*}
    h^{-1} (x)=\left\{
    \begin{aligned}
        & g^{-1}(|x|)\frac{x}{|x|}\ \ \ \ x\neq 0,\\
        & 0 \ \ \ \ \ \ \ \ \ \ \ \ \ \ \ \ \ x=0.
    \end{aligned}
    \right.
\end{equation*}
 Therefore, we can define the inverse transformation map $x \mapsto y=h^{-1}(x)$.

\subsection{Transformed Langevin Diffusion and its discretization}

With the transformed density defined in \eqref{Tdens}, the transformed overdamped Langevin diffusion is given by 
 \begin{align}\label{TLanD}
    dY_t=-\nabla f_h(Y_t)dt+\sqrt{2}dW_t.
 \end{align}
 We denote the density of $Y_t$ by $\rho_t$ for all $t\ge 0$.  The stationary density function for the diffusion given by \eqref{TLanD} is $\pi_h$ as defined in \eqref{Tdens}. 
 We can apply Euler discretization to the transformed overdamped Langevin diffusion in~\eqref{TLanD} and generate a Markov chain $(y_n)_{n\ge 1}$ via the recursion, 
 \begin{align}\label{TULan}
     y_{n+1}=y_n-\gamma_{n+1}\nabla f_h(y_n)+\sqrt{2\gamma_{n+1}}u_{n+1}
 \end{align}
  where $(u_n)$ is a sequence of  independent and identically distributed $d$-dimensional standard Gaussian vectors and $\gamma>0$ is the fixed step size. The Transformed Unadjusted Langevin algorithm (TULA) in order to  generate samples from a heavy-tailed density $\pi$ is given in~\cref{alg:TULA}.
  
  \begin{algorithm}[t]
  \caption{Transformed Unadjusted Langevin Algorithm (\texttt{TULA})}
  \SetAlgoLined
  \SetKwInOut{Input}{Input}
  \SetKwInOut{Output}{Output}
  \Input{Step size $\gamma$ and a sample $y_0$ from a starting density $\rho_0$ }
  \Output{Sequence $x_1,x_2,\cdots$}
  \For{$n=0,1,\cdots$}{ 
   \ $x_n=h(y_n)$  \hspace{.3in} \Comment{apply the inverse transformation} \;
  \ $y_{n+1}\sim \mc{N}(y_n-\gamma \nabla f_h(y_n),2\gamma I_d)$ \Comment{generate samples} \;
  }\label{alg:TULA}
  \end{algorithm}

We use $\nu_n$ to denote the density of the $n$th iterate $x_n$ and $\pi_{\gamma}$ to denote the stationary density of $(x_n)_{n\ge 1}$. Since the step-size $\gamma$ in \cref{alg:TULA} is a constant, there is a bias between $\pi_\gamma$ and $\pi$. For arbitrary accuracy $\epsilon>0$, by choosing small enough step-size $\gamma$ and large enough number of iterations $n$, we can bound the distance between $\nu_n$ and $\pi$ by $\epsilon$ in terms of KL or R\'enyi divergence.


\subsubsection{A Warm-up Example}
Although our main motivation is to sample from densities that have polynomially decaying tails, in this subsection, we provide a warm-up example on sampling from a density that has exponentially decaying tails (see~\eqref{eq:subexp example} for the definition of the potential function) and does not satisfy LSI, by transforming it to satisfy LSI. Towards that goal, we consider the transformation map in~\eqref{Hmap} with the function $g$ defined as 
\begin{equation}\label{eq:trasnformation map for subexp}
   g(r)=\left\{
    \begin{aligned}
    & dr^2, \quad & r\ge R, \\
    & g_{in}(r), \quad & 0\le r\le R,
    \end{aligned}
    \right.
\end{equation}
where $R>0$ is a constant, with $$g_{in}(r)=dRr\exp\left(-\frac{5}{6}+\frac{3}{2}\frac{r^2}{R^2}-\frac{2}{3}\frac{r^3}{R^3}\right).$$
The above form for $g$ is motivated by~\cite[Equation 15]{johnson2012variable}, where they constructed transformation maps to transform densities that are sub-exponential to sub-Gaussian. We also point out that we consider the form of $g$ in~\eqref{eq:trasnformation map for subexp} only for this section, and it should not be confused with the general form~\eqref{G1} considered in the rest of the paper. By an argument similar to the proof of~\cref{PropertyG}, it could be shown that the transformation map defined with $g$ as in~\eqref{eq:trasnformation map for subexp} is a diffeomorphism. 

Now, consider the potential function defined in a piece-wise manner as
    \begin{equation}\label{eq:subexp example}
        f(x)=\left\{
        \begin{aligned}
        &(1+|x|^2)^{\frac{1}{2}}+\frac{1}{2}d\log |x|, \quad &|x|\ge R,\\
        & (1+d^2 g_{in}^{-1}(|x|)^4)^{\frac{1}{2}}+(d-1)\ln \frac{|x|}{g_{in}^{-1}(|x|)}+\log g_{in}'(g_{in}^{-1}(|x|))-\frac{d}{2}\log d-\ln 2, & |x| \in [0,R].
        \end{aligned}
        \right.
    \end{equation}
The corresponding probability density induced by the potential $f$ above has a lighter tail than the one with potential $|x|$. But it has a heavier tail than densities with potentials $|x|^{\varrho}$ for any $\varrho>1$. For the above potential $f$, the transformed potential is given by 
$$
f_h(x)=(1+d^2|x|^4)^{\frac{1}{2}}-\frac{d}{2}\log d-\log 2.
$$

The LSI constant of the density induced by $f_h$ can be studied via the Holley-Stroock Theorem (see~\cref{HS}). We can write 
    \begin{align*}
        f_h(x)&=d|x|^2+\frac{1}{d|x|^2+(1+d^2|x|^4)^{\frac{1}{2}}}-\frac{d}{2}\log d-\log 2 \\
        &:=d|x|^2-\frac{d}{2}\log d-\log 2 + \text{Osc}(x),
    \end{align*}
    where $\text{Osc}(x)\coloneqq\frac{1}{d|x|^2+(1+d^2|x|^4)^{\frac{1}{2}}}$ and is uniformly bounded by $1$. Meanwhile the density corresponding to the potential function $e^{-d|x|^2}+\frac{d}{2}\log d+\log 2$ satisfies LSI with constant ${1}/{2d}$. Therefore $e^{-f_h}$ satisfies LSI with constant $\chlsi={e}/{2d}$. On the other hand, $f_h(x)$ also has  Lipschitz gradients with constant $L_h=O(d)$. Hence, according to \cite{vempala2019rapid} and~\cref{Phithm}, the iteration complexity of TULA for sampling from a density with potential $f$ as in~\eqref{eq:subexp example} is of order $\Tilde{O}({d}/{\epsilon})$ where $\epsilon$ is the error tolerance in KL-divergence. This is to be contrasted with~\cite[Examples 9 and 11]{chewi2021analysis} on using ULA to sample from densities with potentials of the from $|x|^\varrho$ for $\rho \in [1,2]$. Specifically, we note that TULA has better oracle complexity as long as $\rho \in (1,2]$.

\subsection{Transformed Langevin Diffusions as It\^{o} diffusions}
\label{sec:transformationvsIto}
It is worth noting that the transformed diffusion process in~\eqref{TLanD} could also be interpreted in terms of an It\^{o} diffusion. Specifically,  by a direct calculation, the stochastic process $X_t=h(Y_t)$ has the form
\begin{equation}\label{eq:Langevin to X}
    \begin{aligned}
    dX_t&=b(X_t)dt+\sigma(X_t)dW_t,\\
\end{aligned}
\end{equation}
 with $\sigma(x)\coloneqq\sqrt{2}(\nabla h) (h^{-1}(x))$ and
 \begin{align*}
         b(x)&\coloneqq-(\nabla h^T) (h^{-1}(x))(\nabla h) (h^{-1}(x))\nabla f(x)+(\nabla h^T) (h^{-1}(x)) (\nabla \log \det \nabla h)(h^{-1}(x))\\ &\qquad+(\Delta \cdot h) (h^{-1}(x)),
 \end{align*}
 where $\Laplace\cdot h (\cdot)\in \mb{R}^d$ and is defined co-ordinate wise as $(\Laplace\cdot h(x))_i=\Laplace h_i(x)$ for all $i\in \{1,\cdots, d\}
 $ and $x\in \mb{R}^d$. Furthermore, we can actually show that
 \begin{align*}
     b(x)=\frac{1}{2\pi(x)}\langle \nabla , \pi(x) \sigma(x)^T\sigma(x) \rangle,
 \end{align*}
 where $\langle \nabla ,\cdot \rangle$ is the divergence operator for matrix-valued function, i.e $\langle \nabla , \omega(x) \rangle_i=\sum_{j=1}^d \frac{\partial \omega_{i,j}(x)}{\partial x_j} $ for $\omega:\mb{R}^d\to \mb{R}^{d\times d}$.

 The above form of $b(x)$ follows by noting that from the form of $\pi(x)$ in~\eqref{eq:target}, we have
 \begin{align*}
     \frac{1}{2\pi(x)}\langle \nabla , \pi(x) \sigma(x)^T\sigma(x) \rangle&=\frac{1}{2}\langle \nabla,\sigma^T(x)\sigma(x)\rangle -\frac{1}{2}\sigma^T(x)\sigma(x)\nabla f(x) \\
     &=-(\nabla h^T) (h^{-1}(x))(\nabla h) (h^{-1}(x))\nabla f(x)+\frac{1}{2}\langle \nabla,\sigma^T(x)\sigma(x)\rangle.
 \end{align*}
 Meanwhile from~\eqref{Hmap}, based on elementary algebraic manipulations, we obtain that 
 \begin{align*}
     \frac{1}{2}\langle \nabla,\sigma^T(x)\sigma(x)\rangle&=\left[2g''(g^{-1}(|x|))+(d-1)\frac{g'(g^{-1}(|x|))^2}{|x|}-(d-1)\frac{|x|}{g^{-1}(|x|)^2}\right]\frac{x}{|x|},\\
     (\Delta \cdot h)(h^{-1}(x))&=\left[g''(g^{-1}(|x|))+(d-1)\frac{g'(g^{-1}(|x|))}{g^{-1}(|x|)}-(d-1)\frac{|x|}{g^{-1}(|x|)^2}\right]\frac{x}{|x|},
\end{align*}
and
\begin{align*}
     &(\nabla h^T) (h^{-1}(x)) (\nabla \log \det \nabla h)(h^{-1}(x))\\ =&\left[g''(g^{-1}(|x|))+(d-1)\frac{g'(g^{-1}(|x|))^2}{|x|}-(d-1)\frac{g'(g^{-1}(|x|))}{g^{-1}(|x|)}\right]\frac{x}{|x|}.
 \end{align*}

This highlights the fact that transformations provide a way of constructing the drift and diffusion terms in the It\^{o} diffusion that take into account the heavy-tailed nature of the target density. However, it turns out that the results on the analysis of discretizations of It\^{o} diffusion from~\cite{erdogdu2018global, li2019stochastic}, which are in the Wasserstein metric, are not applicable to the class of It\^{o} diffusion of the form above; indeed the stronger Wasserstein contraction conditions made in those works are not satisfied by the above class of It\^{o} diffusions. We leave a detailed investigation of analysis of discretizations of  It\^{o} diffusion above, in stronger KL or R\'enyi metrics, as future work.

\section{Convergence Results}\label{sec:THM}

In this section, we will impose assumptions on the potential function $f$ under which we show exponential ergodicity of the transformed Langevin diffusion and convergence results for \cref{alg:TULA}.  

\subsection{Convergence along the transformed Langevin diffusions} 
We first state convergence results for the continuous time Langevin diffusion under various curvature-related assumptions on the potential function.

\begin{myassump}{A1}\label{ass:A3}
(Dissipativity) There exists $A,B,N_1>0,\alpha\in [1,2]$ such that for all $|x|>N_1$:
\begin{align*}
     f'(\psi(|x|))\psi'(|x|)|x|-b\beta d|x|^\beta +(d-\beta)>A|x|^\alpha-B,
\end{align*}
where $\psi(r)=e^{br^\beta}$ for all $r\ge b^{-\frac{1}{\beta}}$.
\end{myassump}

\noindent \cref{ass:A3} is imposed to guarantee that the transformed potential function satisfies the dissipativity condition. We next recall the dissipativity condition for completeness.
\begin{myassump}{B1}\label{alphadis}
 ($\alpha_h$-dissipativity) We say that the transformed potential function $f_h:\mb{R}^d\to \mb{R}$ satisfies the $\alpha_h$-dissipativity condition with $\alpha_h\in [1,2]$ if there exists $A_h,B_h>0$ such that for all $x\in \mb{R}^d$:
\begin{align*}
    \langle \nabla f_h(x),x \rangle>A_h|x|^{\alpha_h}-B_h.
\end{align*}
\end{myassump}
If the transformed potential function satisfies the $\alpha_h$-dissipativity condition with $\alpha_h=1$, then the corresponding transformed density $\pi_h$ satisfies a Poincar\'e inequality with certain constant $\chpoincare$ depending on the potential function. Then, similar to~\cite{vempala2019rapid}, we obtain the following result. 
\begin{theorem}\label{PITLMC} Assume the initial potential function $f$ satisfies \cref{ass:A0} and \cref{ass:A3} with $\alpha=1$. Then, the transformed density $\pi_h$ with $\beta=1$ and $b\ge \frac{r}{8(d-1)}$ satisfies a Poincar\'e inequality with a constant $\chpoincare$ depending on $f$. Therefore along the transformed Langevin diffusion \eqref{TLanD}, we have for $q \geq 2$ that
\begin{equation*}
    R_{q}(\rho_t|\pi_h)\le \left\{
    \begin{aligned}
    &R_{q}(\rho_0|\pi_h)-\frac{2\chpoincare t}{q}\quad \text{if }R_{q}(\rho_0|\pi_h)\ge 1 \text{ as long as } R_q(\rho_t|\pi_h)\ge 1,\\
    & e^{-\frac{2\chpoincare t}{q}}R_q(\rho_0|\pi_{h}) \quad\ \  \text{if }R_{q}(\rho_0|\pi_h)\le 1.
    \end{aligned}
    \right.
\end{equation*}
\end{theorem}

\begin{myassump}{A2}\label{ass:A5}(Degenerate convexity) 
There exists $\mu, N_2>0,\theta\ge 0$ such that for all $|x|>N_2$: 
\begin{align*}
    &f'(\psi(|x|))\psi'(|x|)|x|^{-1}-b\beta d|x|^{\beta-2}+(d-\beta)|x|^{-2}>\frac{\mu}{(1+\frac{1}{4}|x|^2)^{\theta/2}},\\
    &f''(\psi(|x|))\psi'(|x|)^2+f'(\psi(|x|))\psi''(|x|)-b \beta(\beta-1)d|x|^{\beta-2}-(d-\beta)|x|^{-2}>\frac{\mu}{(1+\frac{1}{4}|x|^2)^{\theta/2}}.
\end{align*}
where $\psi(r)=e^{br^\beta}$ for all $r\ge b^{-\frac{1}{\beta}}$.
\end{myassump}

\noindent \cref{ass:A5} is imposed to guarantee the transformed potential function is degenerately convex at infinity. We now recall the definition of degenerate convexity at infinity from \cite{erdogdu2020convergence}.
\begin{myassump}{B2}\label{ass:Degenerate convexity at infinity}(Degenerate convexity at infinity)
 We say that the transformed potential function $f_h:\mb{R}^d\to \mb{R}$ is degenerately convex at infinity if there exist a function $\Tilde{\phi}:\mb{R}^d\to \mb{R}$ such that for a constant $\xi_h\ge 0$
\begin{align*}
     \lv f_h-\Tilde{\phi} \rv_{\infty}\le \xi_h,
\end{align*}
where $\Tilde{f}$ satisfies,
\begin{align*}
    \nabla^2 \Tilde{f}(x)\succeq \frac{\mu_h}{(1+\frac{1}{4}|x|^2)^{\theta_h/2}}I_d,
\end{align*}
for some $\mu_h>0$ and $\theta_h\ge 0$.
\end{myassump}
\noindent The degenerate convexity at infinity condition is weaker than the strong convexity at infinity. If a potential function satisfies degenerate convexity at infinity, then the corresponding probability measure satisfies m-LSI. Similar to~\cite{toscani2000trend}, we obtain the following result.  
\begin{theorem}\label{mLSILMC} Assume the initial potential function $f$ satisfies \cref{ass:A0} and \cref{ass:A5}. Then the transformed density $\pi_h$ satisfies a modified Logarithmic Sobolev Inequality with a uniform constant $\delta$ (see~\eqref{eq:mLSI}) and constant $\chmlsi$ depending on $f$. Therefore along the transformed Langevin diffusion \eqref{TLanD}, we have
\begin{align*}
    H_{\pi_h}(\rho_t)\le \frac{C}{t^\ell},
\end{align*}
where the constant $C$ depends on the potential $f$ and the transformation $h$ and $\ell=(1-2\delta)/\delta$.
\end{theorem}

\begin{remark}
Note that the above rate is faster than any polynomial but not truly exponential. While the above rate could be made exponential with additional assumptions on the tail and/or assumptions on the initial distribution, we do not present such modifications here.
\end{remark}

\begin{myassump}{A3}\label{ass:A1}(Strong convexity at infinity) There exists $N_3,\rho>0$ such that for all $|x|>N_3$: 
\begin{align*}
    &f'(\psi(|x|))\psi'(|x|)|x|^{-1}-b \beta d|x|^{\beta-2}+(d-\beta)|x|^{-2}>\rho,\\
    &f''(\psi(|x|))\psi'(|x|)^2+f'(\psi(|x|))\psi''(|x|)-b\beta(\beta-1)d|x|^{\beta-2}-(d-\beta)|x|^{-2}>\rho,
\end{align*}
where $\psi(r)=e^{br^\beta}$ for all $r\ge b^{-\frac{1}{\beta}}$.
\end{myassump}

\noindent \cref{ass:A1} is imposed to guarantee that the transformed potential function is strongly convex with parameter $\rho_h$ at infinity. The property that a potential function is strongly convex at infinity implies that the corresponding probability measure satisfies a LSI with a  certain parameter depending on the potential function and the transformation map. 

\begin{theorem}\label{LSILMC} Assume the initial potential function $f$ satisfies \cref{ass:A0} and \cref{ass:A1}, then the transformed density $\pi_h$ satisfies a logarithmic Sobolev inequality with a constant $\chlsi$ depending on $f$. Therefore along the transformed Langevin diffusion \eqref{TLanD}, we have
\begin{align*}
    H_{\pi_h}(\rho_t)\le e^{-2t \chlsi }H_{\pi_h}(\rho_0).
\end{align*}
\end{theorem}

\subsection{Convergence along TULA} \label{Convergence}
In this section, we state two types of convergence results for  \cref{alg:TULA}, based on~\cref{Phithm} and \cite{vempala2019rapid,chewi2021analysis}. While the works of \cite{vempala2019rapid,chewi2021analysis} provide results only for exponentially decaying densities, our results below are applicable for polynomially-decaying densities based on the constructed transformation maps. To proceed, we first list smoothness conditions on the potential function $f$. 

\begin{myassump}{A4}\label{ass:A2}(Gradient Lipschitz) There exists $N_4,L>0$ such that for all $|x|>N_4$: 
\begin{align*}
    &f'(\psi(|x|))\psi'(|x|)|x|^{-1}-b \beta d|x|^{\beta-2}+(d-\beta)|x|^{-2}<L,\\
    &f''(\psi(|x|))\psi'(|x|)^2+f'(\psi(|x|))\psi''(|x|)-b\beta(\beta-1)d|x|^{\beta-2}-(d-\beta)|x|^{-2}<L
\end{align*}
where $\psi(r)=e^{br^\beta}$ for all $r\ge b^{-\frac{1}{\beta}}$.
\end{myassump}

\noindent \cref{ass:A2} is imposed to guarantee that the transformed potential function has Lipschitz gradients with parameter $L_h$. Such smoothness conditions on the potential function are required to study the discrete Markov chain generated in the unadjusted Langevin algorithm. We also remark that it is possible to relax the Lipschitz gradient assumption to certain weak-smooth conditions on the gradient; we do not pursue such extensions in this work. While~\cref{mLSILMC} holds under m-LSI, to get the corresponding result for~\cref{alg:TULA}, we also require the following additional tail-conditions.
\begin{myassump}{A5}{(Tail assumption)}\label{ass:tail assumption on original potential} For some $m\ge 0$, $\alpha_1\in [0,1]$ and $N_5>0$, there exists a positive constant $\ctail^*$ such that for all $\lambda\ge N_5$,
\begin{align*}
    \pi\left\{ |\cdot|\ge m+\lambda \right\}\le 2\exp\left( -\left( \frac{\psi^{-1}(\lambda)}{\ctail^*} \right)^{\alpha_1} \right),
\end{align*}
where $\psi(r)=e^{br^\beta}$ for all $r\ge b^{-\frac{1}{\beta}}$.
\end{myassump}
\begin{myassump}{B5}\label{ass:tail assumption for transformed density} For some $m_h\ge 0$ and $\alpha_{h,1}\in [0,1]$, there exists a positive constant $\chtail$ such that for all $\lambda\ge 0$,
\begin{align*}
    \pi_h\left\{ |\cdot|\ge m_h+\lambda \right\}\le 2\exp\left( -\left( \frac{\lambda}{\chtail} \right)^{\alpha_{h,1}} \right).
\end{align*}

\begin{theorem}\label{thm:mlsi and tail ass} In addition to the assumptions in~\cref{mLSILMC}, assume that the initial potential $f$ is such that $\nabla f_h(0)=0$, and it satisfies \cref{ass:A2} and \cref{ass:tail assumption on original potential}. Furthermore, let $\epsilon^{-1},m_h,\chmlsi,\chtail,L_h,R_2(\rho_0|\hat{\pi}_h)\ge 1$ ($\hat{\pi}_h$ is as defined in \eqref{eq:modified density} with $\hat{R}=2\int_{\mb{R}^d} |x| \pi_h(x)dx$ and $\hat{\gamma}=(3072 n\gamma)^{-1}$), and $m_h,\chtail,R_2(\rho_0|\pi)\le d^{\Tilde{O}(1)}$. Then, \cref{alg:TULA} with an step size 
\begin{align*}
    \gamma=\Tilde{\Theta}\left( \frac{\epsilon}{d\chmlsi^2\chtail^{\theta}L_h^2 R_{2q}(\nu_0|\pi)^{\theta/\alpha_{h,1}}} \times \min\left\{ 1, \frac{1}{q\epsilon}, \frac{d}{m_h}, \frac{d}{R_2(\rho_0|\hat{\pi}_h)}, \left( \frac{R_{2q}(\nu_0|\pi)^{1/\alpha_{h,1}}}{m_h} \right)^{\theta}  \right\} \right),
\end{align*}
satisfies $R_q(\nu_n|\pi)\le \epsilon$, for all $q\ge 2$ after 
\begin{align*}
    n=\Tilde{\Theta}\left( \frac{d R_{2q}(\nu_0|\pi)^{2\theta/\alpha_{h,1}} \chmlsi^{4} \chtail^{2\theta}L_h^{2}   }{\epsilon} \max\left\{ 1,\epsilon, \frac{m_h}{d}, \frac{R_2(\rho_0|\hat{\pi}_h)^{1/2}}{d}, \left( \frac{m_h}{R_{2q}(\nu_0|\pi)^{1/\alpha_{h,1}}} \right)^{\theta} \right\} \right)
\end{align*}
iterations, for some $\theta \in [0,1]$ (depending on the parameter $\delta$ in~\eqref{eq:mLSI}). Explicit form of $\chmlsi$ is the constant $\lambda$ in \eqref{mLSIconst} and $m_h,\chtail,L_h$ are given in \eqref{eq:tail ass para m},\eqref{eq:tail ass para ctail}, \eqref{GLpara} respectively. The $\Tilde{\Theta}(\cdot)$ notation hides polylogarithmic factors as well as constants depending on $\theta,q$. 
\end{theorem}

\begin{remark}
In order to obtain a direct quantitative bound, it is important to obtain a control of $R_2(\rho_0|\hat{\pi}_h)$ and $R_{2q}(\nu_0|\pi)$. We refer to~\cite[Section A]{chewi2021analysis} for a proof that the conditions required on $R_2(\rho_0|\hat{\pi}_h)$ is satisfied, and for obtaining a control on the term $R_{2q}(\nu_0|\pi)$.
\end{remark}

\begin{theorem}\label{LSITULA} In addition to the assumptions in \cref{LSILMC}, assume that $f$ satisfies \cref{ass:A2}. Then \cref{alg:TULA}, for any $y_0\sim \rho_0$ with $H_{\pi_h}(\rho_0)<\infty$, and with step size 
$$0<\gamma\le \frac{1}{2L^2_h \chlsi}\min\left\{1,\frac{\epsilon}{4d} \right\},$$ 
satisfies  $H_{\pi}(\nu_n)<\epsilon$, for any $\epsilon>0$ after $$n = \tilde{\Theta}\left(\frac{\chlsi}{2 \gamma}\log \frac{2H_{\pi_h}(\rho_0)}{\epsilon}\right)$$ iterations. Explicit forms of $\chlsi$ and $L_{h}$ are given in \eqref{LSIConst} and \eqref{GLpara}.
\end{theorem}
\begin{remark}
As argued in~\cite{vempala2019rapid}, if we let $\rho_0$ to be a Gaussian distribution with mean being any stationary point of $f_h$ and covariance matrix being $(1/L_h)I_d$, then $H_{\pi_h}(\rho_0) = \tilde{O}(d)$. Furthermore, we also remark that similar convergence results in the stronger R\'enyi metric, for all $q\geq 4$ holds via~\cite[Theorem 4]{chewi2021analysis}.
\end{remark}

\begin{remark}
We leave a detailed study of obtaining convergence results for the underdamped Langevin dynamics and its discretization as future work.
\end{remark}

 \section{Relation with Poincar\'{e} Inequalities based on Non-local Dirichlet Forms}\label{sec:htfunctionalineq}

We now discuss the relationship between our assumptions on the potential function and functional inequalities like super and weak Poincar\'e inequalities that arise in characterizing the heavy-tailed stationary distributions of certain $\vartheta$-stable driven diffusions~\cite{rockner2001weak,rockner2003harnack,cattiaux2010functional, wang2014simple, wang2015functional,huang2021approximation}. Recall from~\cref{sec:prelim} that the Dirichlet form associated with Langevin diffusion in~\eqref{LDy} is of the form $\mc{E}(\phi)  = \int |\nabla \phi(x)|^2 \mu(x) dx $. However, in the case of $\vartheta$-stable driven diffusions the corresponding non-local Dirichlet form is given by
\begin{align}\label{eq:htDirichlet}
    \mc{E}(\phi)&:=\iint_{\{x\neq y\}} \frac{(\phi(x)-\phi(y))^2}{|x-y|^{(d+\vartheta)}} dx \mu(dy),
 \end{align}
for all functions in the Dirichlet domain  $\mc{D}(\mc{E})$; see for example~\cite{wang2014simple}. We now introduce similar functional inequalities that are associated with stable-driven diffusions. 
\begin{definition}[Poincar\'{e}-type Inequalities]
A Markov triple $(\mb{R}^d,\mu,\Gamma)$ (with $\mu$ a probability measure), with the Dirichlet form as in~\eqref{eq:htDirichlet} is said to satisfy 
    \begin{itemize}[leftmargin=0.2in]
        \item a \textit{Poincar\'{e} inequality} if for any function $\phi:\mb{R}^d\to \mb{R}$ in the Dirichlet domain $\mc{D}(\mc{E})$ and $C>0$,
    \begin{align*}
    \text{Var}_{\mu}(\phi)\le C \mc{E}(\phi),
    \end{align*}
    \item a \textit{weak Poincar\'{e} inequality} if for any function $\phi:\mb{R}^d\to \mb{R}$ in the Dirichlet domain $\mc{D}(\mc{E})$ and $r>0$,
    \begin{align*}
    \text{Var}_{\mu}(\phi)\le \alpha(r) \mc{E}(\phi)+r\lv \phi \rv_{\infty}^2,
    \end{align*}
    \item a \textit{super Poincar\'{e} inequality} if for any function $\phi:\mb{R}^d\to \mb{R}$ in the Dirichlet domain $\mc{D}(\mc{E})$ and $r>0$,
    \begin{align*}
    \mu(\phi^2)\le r \mc{E}(\phi)+\beta (r) \mu(|\phi|)^2,
    \end{align*}
    where $\mu(\varphi)=\int_{\mb{R}^d} \varphi(x) \mu(dx) $ for all $\varphi\in L^{1}(\mu)$.
    \end{itemize}
\end{definition}
In the following, we will discuss the relation between \cref{ass:A3}, \cref{ass:A1}, and \cref{ass:A5}, and the Poincar\'{e}-type inequalities above. In what follows, the terms $\alpha(r)$ and $\beta(r)$ are defined as
    \begin{align}\label{eq:alphar}
        \alpha(r)=\inf_{s>0} \left\{\frac{1}{\inf_{0<|x-y|\le s}[(e^{f(x)}+e^{f(y)})|x-y|^{-(d+\vartheta)}] }: \int\int_{|x-y|>s} e^{-f(x)}e^{-f(y)}dxdy\le r/2 \right\},
    \end{align}
    \begin{align}\label{eq:betar}
        \beta(r)=\inf_{t,s>0}\left\{ \frac{2\mu(\omega)}{\inf_{|x|\ge t}\omega(x)}+\beta_t(t\wedge s): \frac{2}{\inf_{|x|\ge t} \omega(x)}+s\le r \right\},
    \end{align}
    where for any $t>0$, we have
    \begin{align*}
        \beta_t(s)=\inf_{u>0}\left\{ \frac{ (\sup_{|z|\le 2t} e^{f(z)})^2 }{u^d (\inf_{|z|\le t} e^{f(z)})}: \frac{ u^{\vartheta} (\sup_{|z|\le 2t} e^{f(z)} )}{(\inf_{|z|\le t} e^{f(z)} ) }\le s \right\}.
    \end{align*}
    The function $\omega$ will depend on the properties of the potential $f$.

\begin{proposition}\label{Prop:A3toPoincare} If the original potential satisfies \cref{ass:A3} with parameters $\alpha,A,B$, then 
\begin{enumerate}
    \item [(1)] If $\alpha>\beta$ or $\alpha=\beta,\vartheta<A\beta^{-1}b^{-1}$, the original density function satisfies the super Poincar\'{e} inequality with 
    $$
    \omega(x)=\frac{C}{2^{d+\vartheta}}|x|^{A\alpha^{-1}b^{-\frac{\alpha}{\beta}}\log^{\frac{\alpha}{\beta}-1}(|x|)-\vartheta}\log^{-\frac{B}{\beta}}(|x|),
    $$ 
    for some positive constant C. 
    \item [(2)] If $\alpha=\beta, \vartheta\ge A\beta^{-1}b^{-1}$, the original density function satisfies the weak Poincar\'{e} inequality. 
\end{enumerate}
\end{proposition}

\begin{proposition}\label{Prop:A5toPoincare} If the original potential satisfies \cref{ass:A5} with parameters $\mu,\theta$, then 
\begin{enumerate}
    \item [(1)] If $\theta<2-\beta$ or $\theta=2-\beta,\vartheta<\mu\beta^{-1} b^{-1}$, the original density function satisfies the super Poincar\'{e} inequality with $\omega(x)$ defined as
    \begin{equation*}
    \omega(x)=\left\{
        \begin{aligned}
        & \frac{C}{2^{d+\vartheta}}|x|^{(1-\theta)^{-1}(2-\theta)^{-1}\mu b^{-\frac{2-\theta}{\beta}}\log^{\frac{2-\theta}{\beta}-1}(|x|)+1-(d+\vartheta)}\log^{-\frac{d-\beta}{\beta}}(|x|),\quad \theta<2-\beta,\\
        &\frac{C}{2^{d+\vartheta}}|x|^{b^{-\frac{2-\theta}{\beta}}\log^{\frac{2-\theta}{\beta}-1}(|x|)-\vartheta}\log^{-\frac{d-\beta}{\beta}}(|x|), \quad \theta=2-\beta,\vartheta<\mu\beta^{-1} b^{-1}.
        \end{aligned}
        \right.
    \end{equation*}
    where $C$ is some positive constant.
    \item [(2)] If $\theta=2-\beta,\vartheta\ge \mu \beta^{-1} b^{-1}$ or $\theta>2-\beta$, the original density function satisfies the weak Poincar\'{e} inequality.
    \begin{align*}
        \alpha(r)=\inf \left\{\frac{1}{\inf_{0<|x-y|\le s}[(e^{f(x)}+e^{f(y)})|x-y|^{-(d+\vartheta)}] }: \int\int_{|x-y|>s} e^{-f(x)}e^{-f(y)}dxdy\le r/2 \right\}.
    \end{align*}
\end{enumerate}
\end{proposition}

\begin{proposition}\label{Prop:A1toPoincare} If the original potential function satisfies \cref{ass:A1} with parameter $\rho$, then 
\begin{enumerate}
    \item [(1)] If $\beta\in (1,2)$ or $\beta=2,\vartheta<\frac{1}{2}\rho b^{-1}$, the original density function satisfies the super Poincar\'{e} inequality with 
    $$
    \omega(x)=\frac{C}{2^{d+\vartheta}}|x|^{\frac{1}{2}\rho b^{-\frac{2}{\beta}}\log^{\frac{2}{\beta}-1}(|x|)-\vartheta}\log^{-\frac{d-\beta}{\beta}}(|x|),
    $$ 
    for some positive constant $C$. 
    \item [(2)] If $\beta=2, \vartheta=\frac{1}{2}\rho b^{-1}, d=1,2$, the original density function satisfies the Poincar\'{e} inequality. 
    \item [(3)] If $\beta=2,\vartheta=\frac{1}{2}\rho b^{-1},d\ge 3$ or $\beta=2,\vartheta>\frac{1}{2}\rho b^{-1}$, the original density function satisfies the weak Poincar\'{e} inequality. 
\end{enumerate}
\end{proposition}

\begin{remark}\label{Remark:A1A3toPoincare}
For the example of multivariate $t$-distribution, it is shown later in \cref{lem:tdistribution}, that it satisfies \cref{ass:A1} with $\beta=2$, $\alpha=2$, $A=2b \kappa $, $B\ge 0$ and arbitrary $\mu\in (0,2b\kappa)$. Therefore when $\kappa>\vartheta$, it falls into the class of densities described by the super Poincar\'e inequality. When $0<\kappa\le \vartheta$, it falls into the class of densities described by the weak Poincar\'e inequality. This classification of the multivariate $t$-distributions with different degrees of freedom coincides with \cite[Corollary 1.2]{wang2015functional}. 
\end{remark}
\begin{remark}\label{Remark:A5toPoincare}
For the multivariate $t$-distribution with degrees of freedom $\kappa$, we show later in \cref{lem:tdistribution} that it satisfies \cref{ass:A5} with arbitrary $\mu\in(0, b\kappa\beta(\beta-1) )$ and $\theta=2-\beta$. When $\vartheta<\kappa(\beta-1)$, we can show that multivariate $t$-distribution with $\kappa$ degrees of freedom satisfies the super Poincar\'e inequality which agrees with the results in \cite{wang2014simple} and our \cref{Remark:A1A3toPoincare} above.
\end{remark}

\section{Illustrative Examples}\label{sec:Examples}
In this section, we introduce a specific transformation map $h$ defined by \eqref{G1} and \eqref{Hmap} with $\beta=2$ and $g_{in}$ defined by the following equation. For all $r\le b^{-\frac{1}{2}}$,
\begin{align}\label{Ginbeta2}
    g_{in}(r)=r b^{\frac{1}{2}}\exp\left(br^2-\frac{10}{3}b^{\frac{3}{2}}r^3+\frac{15}{4}b^2r^4-\frac{6}{5}b^{\frac{5}{2}}r^5+\frac{47}{60}\right).
\end{align}
Using the above transformation map, we analyze the oracle complexity of TULA for sampling from the multivariate $t$-distribution and related densities. 

\subsection{Example 1} \label{sec:texample}
The density and potential function of the multivariate $t$-distribution are respectively given by
\begin{align}\label{eq:tdistribution}
    \pi(x)\propto (1+|x|^2)^{-\frac{d+\kappa}{2}}, \quad f(x)=\frac{d+\kappa}{2}\log (1+|x|^2),
\end{align}
where $\kappa$ is the degrees of freedom parameter. We first show that the above $g_{in}$ satisfies~\cref{ass:G0} and hence the corresponding $h$ is a diffeomorphism. 
\begin{lemma}\label{GinG0} With $g_{in}$ defined in \eqref{Ginbeta2}, $\beta=2$ and $f(x)=\frac{d+\kappa}{2}\log (1+|x|^2)$, $g$ defined in \eqref{G1} satisfies \cref{ass:G0}.
\end{lemma}

Next, we show that the potential function of the multivariate $t$-distribution satisfies the assumptions we introduced in~\cref{sec:THM}.
\begin{lemma}\label{lem:tdistribution}
We have for the following for the potential function $f(x)$ in~\eqref{eq:tdistribution}:
\begin{itemize}
    \item [(1)] $f(x)$ is isotropic and $f\in \mc{C}^2(\mb{R}^d)$;
    \item [(2)] $f$ satisfies \cref{ass:A2} with some $N_4>0$ and $L=2\kappa b^{\frac{2}{\beta}}\beta$;
    \item [(3)] $f$ satisfies \cref{ass:A3} with $\alpha=\beta$, $A=\kappa b \beta$ and some $B\ge 0$,$N_1>0$. 
    \item [(4)] $f$ satisfies \cref{ass:A5} with arbitrary $\mu\in(0,\kappa b\beta(\beta-1))$, $\theta=2-\beta$ and some $N_2>0$.
\end{itemize}
\end{lemma}
Hence, we can apply \cref{LSITULA} with $n= \Tilde{O}(L^2_h \chlsi^2 d/\epsilon)$, where $\chlsi$ and $L_h$ are two constants that depend on $f$ as introduced in \eqref{LSIConst} and \eqref{GLpara}. However, the dependence of $\chlsi$ and $L_h$ on $f$ would affect the order of $n$ significantly, especially in terms of the dimension parameter.
Specifically, after explicitly calculating the constants $\chlsi$ and $L_h$, the mixing time of TULA in KL-divergence with error tolerance $\epsilon$ is of order $n=\Tilde{O}(\exp(2d) d^{d+1}\epsilon^{-1})$. A detailed proof of \cref{lem:tdistribution} and the calculation for order estimation of the mixing time $n$ are given in Sections \ref{sec:secegproofs} and~\ref{sec:orderest} respectively.

Despite the above result for the multivariate $t$-distribution, we next demonstrate through several examples that as long as the tail becomes slightly lighter, we get linear dependency on both the dimension parameter and inverse of the target accuracy parameter. In the next several examples, we use the following result form~\cite{chen1997estimates} to calculate the LSI constant. Furthermore, following a similar argument in the proof of~\cref{lem:tdistribution}, one can show that the potentials satisfy the assumptions required by~\cref{LSITULA}. However, for simplicity, we directly calculate the LSI constants of the transformed potential and use the result from~\cite{vempala2019rapid}. 

\begin{corollary}\label{Cor:estimate of clsi}[Simplified version of \cite[Corollary 1.4]{chen1997estimates}] For the Langevin diffusion process with generator $\mc{L}=-\nabla f \cdot \nabla+\Laplace$,  let $\lambda_f(x)$ be the largest eigenvalue of the matrix $\nabla^2 f(x)$ and let $\bar{\beta}(r)=\inf_{|x|\ge r}\{-\lambda_f(x)\}$. If $\sup_{r\ge 0} \bar{\beta}(r)>0$, then the stationary measure to this Langevin diffusion satisfies LSI with constant $2/\alpha(\mc{L})$ such that
\begin{align*}
    \alpha(\mc{L})\ge \frac{2}{a_0^2}\exp\left(1-\int_0^{a_0} r\bar{\beta}(r)dr\right)>0,
\end{align*}
where $a_0>0$ is the unique solution to the equation $\int_0^a \bar{\beta}(r)dr=2/a$.
\end{corollary}

\subsection{Example 2}
The next potential function $f$ we consider is given by 
    \begin{equation*}
        f(x)=\left\{
        \begin{aligned}
        &\frac{d+\kappa}{2}\log (1+|x|^2)-\frac{d+\kappa}{2}\log (1+|x|^{-2}) \\
        &(\upsilon_f d+1)\log\log|x|+\left(\upsilon_f+\frac{1}{2}\right)d\log (1+2b(\log |x|)^{-1}) &|x|\ge e,\\
        &(d-1)\log|x|+\frac{d}{2}\log g_{in}^{-1}(|x|)^2+\left(\frac{1}{2}+\upsilon_f\right)d\log\left(1+\frac{1}{2}g_{in}^{-1}(|x|)^2\right)\\
        &\-(d-1)\log g_{in}^{-1}(|x|) +\log g_{in}'(g_{in}^{-1})(|x|)+\upsilon_f d\log b+\left[\left(\frac{1}{2}+\upsilon_f\right)d-1\right]\log 2, &0\le |x|<e.
        \end{aligned}
        \right.
    \end{equation*}
    where $\upsilon_f \in (-\frac{3}{2},\frac{15}{2})$. With the transformation $h$ defined by \eqref{G1}, \eqref{Hmap} and \eqref{Ginbeta2} and $b=\frac{d}{2\kappa}$, the transformed potential is
    \begin{align*}
        f_h(x)=\frac{d}{2}|x|^2+\left(\frac{1}{2}+\upsilon_f\right)d\log\left(1+\frac{1}{2}|x|^2\right)+\upsilon_f d\log b+\left[\left(\frac{1}{2}+\upsilon_f\right)d-1\right]\log 2, \quad\forall x\in \mb{R}^d.
    \end{align*}
    We can find the LSI constant of the transformed density $\pi_h\propto e^{-f_h(x)}$ by \cite[Corollary 1.4]{chen1997estimates}. First, note that the two eigenvalues of $\nabla^2 f_h(x)$ are 
    \begin{align*}
        &\lambda_1(x)=d \left[1+\left(\frac{1}{2}+\upsilon_f\right)\frac{1}{1+\frac{1}{2}|x|^2}\right],~\quad~\text{and}~\quad\lambda_2(x)=d\left[1+\left(\frac{1}{2}+\upsilon_f\right)\frac{1-\frac{1}{2}|x|^2}{(1+\frac{1}{2}|x|^2)^2}\right].
    \end{align*}
    We now consider the following cases.
    \begin{itemize}
        \item [(a)] When $\upsilon_f=-\frac{1}{2}$: $\lambda_1(x)=\lambda_2(x)=d$. The LSI constant $\chlsi=\frac{2}{d}$. 
        \item [(b)] When $\upsilon_f\in (-\frac{1}{2},\frac{15}{2})$, $\lambda_2(x)<\lambda_1(x)$ for all $x\in \mb{R}^d$. Therefore 
        \begin{equation*}
            \Bar{\beta}(r)=\left\{
            \begin{aligned}
            &\left[1-\frac{1}{8}\left(\upsilon_f+\frac{1}{2}\right)\right]d, &0\le r\le \sqrt{6},\\
            &\left[1+\frac{1-\frac{1}{2}r^2}{(1+\frac{1}{2}r^2)^2}\left(\frac{1}{2}+\upsilon_f\right)\right]d, &r>\sqrt{6}.
            \end{aligned}
            \right.
        \end{equation*}
        and 
        \begin{align*}
            \int_0^{a_0}\Bar{\beta}(r)dr=\frac{2}{a_0} \quad  \implies \quad a_0=\left(\frac{2}{1-\frac{1}{8}(\upsilon_f+\frac{1}{2})}d^{-1}\right)^{\frac{1}{2}}. 
        \end{align*}
        The LSI constant is hence given by
        \begin{align*}
            \chlsi=a_0^2\exp\left(\int_0^{a_0} r\Bar{\beta}(r)dr-1\right)=\frac{2}{1-\frac{1}{8}(\upsilon_f+\frac{1}{2})}d^{-1}.
        \end{align*}
        \item [(c)] When $\upsilon_f\in (-\frac{3}{2},-\frac{1}{2})$, $\lambda_1(x)<\lambda_2(x)$ for all $x\in \mb{R}^d$. Therefore 
        \begin{align*}
            \Bar{\beta}(r)=\inf_{|x|>r} \lambda_1(x)=\lambda_1(0)=(\frac{3}{2}+\upsilon_f) d.
        \end{align*}
        and 
        \begin{align*}
            \int_0^{a_0}\Bar{\beta}(r)dr=\frac{2}{a_0} \quad  \implies \quad a_0=\left(\frac{2}{\frac{3}{2}+\upsilon_f}d^{-1}\right)^{\frac{1}{2}}. 
        \end{align*}
        The LSI constant is 
        \begin{align*}
            \chlsi=a_0^2\exp(\int_0^{a_0} r\Bar{\beta}(r)dr-1)=\frac{2}{\frac{3}{2}+\upsilon_f}d^{-1}.
        \end{align*}
    \end{itemize}
   Hence, we have that $\chlsi=O(d^{-1})$. Combined with the fact that the gradient Lipschitz constant of $f_h$ is $L_h=O(d)$, according to \cite{vempala2019rapid}, the iteration complexity to achieve $\epsilon$  error tolerance in KL-divergence is of order $\Tilde{O}({d}/{\epsilon})$, where $\tilde{O}$ hides numerical constants and poly-logarithmic factors.

\subsection{Example 3}\label{sec:Heavytailexample}
The next potential function is given by 
\begin{equation*}
    f(x) = \left\{
    \begin{aligned}
    &d(1+\frac{1}{2b})\log |x|+(\frac{d}{2}+1)\log\log|x|+d\log(1+2b(\log|x|)^{-1})\\ &-(d-1)\log 2-\frac{d}{2}\log b &|x|>e \\
    &(d-1)\log |x|-(d-1)\log g_{in}^{-1}(|x|) +\frac{d}{2}g_{in}^{-1}(|x|)^2+d\log (1+\frac{1}{2}g_{in}^{-1}(|x|)^2) \\&
             +\log g_{in}'(g_{in}^{-1})(|x|), & 0\le |x|\le e
\end{aligned}
           \right.
\end{equation*}

As a point of reference, we compare  the potential above to the potential function $\Tilde{f}(x)=d(1+\frac{1}{2b})\log (1+|x|)+(\frac{d}{2}+1)\log\log(e+|x|)$. According to \cite{wang2014simple}, if $b={d}/{2\vartheta}$, $\Tilde{f}$ satisfies the weak Poincar\'e inequality with $\vartheta$ being the degree of freedom. The corresponding transformed potential is then given by 
$$
f_h(x)=\frac{d}{2}|x|^2+d\log\left(1+\frac{1}{2}|x|^2\right).
$$ 
The density function induced by this potential function satisfies the LSI and the log-concavity assumption. This follows from by calculating the two eigenvalues of the Hessian matrix $\nabla^2 f_h(x)$, that are given by
    \begin{align*}
        \lambda_1=d\left[1+\frac{1}{1+\frac{1}{2}|x|^2}\right], \qquad~\text{and}\quad \lambda_2=d \left[1+\frac{1-\frac{1}{2}|x|^2}{(1+\frac{1}{2}|x|^2)^2}\right].
    \end{align*}
    For all $x\in \mb{R}^d$, we have that $0<\lambda_i\le 2d$ for $i=1,2$. Therefore the transformed potential $f_h$ is gradient Lipschitz with parameter $2d$. To find the LSI parameter we use  \cite[Corollary 1.4]{chen1997estimates}. For all $x\in \mb{R}^d$: $\lambda_2\le \lambda_1$. 
    \begin{equation*}
        \bar{\beta}(r)=\inf_{|x|>r} \lambda_2=\left\{
        \begin{aligned}
        &\frac{7}{8}d, \qquad &r\in (0,\sqrt{6}],\\
        &(1+\frac{1-\frac{1}{2}r^2}{(1+\frac{1}{2}r^2)^2})d, & r\in (\sqrt{6},\infty).
        \end{aligned}
        \right.
    \end{equation*}
    The solution to the equation $\int_0^a \bar{\beta}(r)dr=2/a$ is given by $a_0=\sqrt{{16}/{7d}}$. The LSI constant hence satisfies
    \begin{align*}
        \chlsi\le a_0^2 \exp\left(\int_0^{a_0} r\bar{\beta}(r)dr-1 \right)=\frac{16}{7d}.
    \end{align*}
    According to \cite{vempala2019rapid}, the iteration complexity is of order $\Tilde{O}({d}/{\epsilon})$, where $\tilde{O}$ hides only numerical constants and poly-logarithmic factors. 

    \subsection{Example 4}
  Our next potential function is given by
  \begin{equation*}
        f(x)=\left\{
    \begin{aligned}
    &d(1+\frac{1}{2b})\log |x|+\log\log|x|+\frac{d}{2}\log(1+2b(\log|x|)^{-1})\\&-(\frac{d}{2}-1)\log 2 &~|x|>e \\
    &(d-1)\log |x|-(d-1)\log g_{in}^{-1}(|x|)+\frac{d}{2}g_{in}^{-1}(|x|)^2+\frac{d}{2}\log (1+\frac{1}{2}g_{in}^{-1}(|x|)^2)\\&+\log g_{in}'(g_{in}^{-1})(|x|), &~ 0\le |x|\le e
    \end{aligned}
    \right.
\end{equation*}

    To study the tail-behavior of the original potential function $f$, we compare it to another potential function $\Tilde{f}(x)=d(1+\frac{1}{2b})\log (1+|x|)+\log\log(e+|x|)$. According to \cite{wang2014simple}, if $b=\frac{d}{2\vartheta}$, $\Tilde{f}$ satisfies the weak Poincar\'e inequality with $\vartheta$ being the degree of freedom. But compare to the previous example, it has a heavier tail because $1<\frac{d}{2}+1$.
    
 The transformed potential in this case is given by  
 $$
 f_h(x)=\frac{d}{2}|x|^2+\frac{d}{2}\log\left(1+\frac{1}{2}|x|^2\right).
 $$ 
 Similar to the previous example, the corresponding density function satisfies LSI and log-concavity assumption. The two eigenvalues of the Hessian matrix are:
    \begin{align*}
        \lambda_1=d \left[1+\frac{1}{2}\frac{1}{1+\frac{1}{2}|x|^2}\right],
\qquad\text{and}\quad         \lambda_2=d \left[1+\frac{1}{2}\frac{1-\frac{1}{2}|x|^2}{(1+\frac{1}{2}|x|^2)^2}\right].
    \end{align*}
    For all $x\in \mb{R}^d$, $0<\lambda_i\le \frac{3}{2}d$ for $i=1,2$. Therefore the transformed potential $f_h$ is gradient Lipschitz with parameter $\frac{3}{2}d$. To find the LSI parameter we use \cite[Cororllary 1.4]{chen1997estimates}. For all $x\in \mb{R}^d$: $\lambda_2\le \lambda_1$. Furthermore, we have
    \begin{equation*}
        \bar{\beta}(r)=\inf_{|x|>r} \lambda_2=\left\{
        \begin{aligned}
        &\frac{15}{16}d, \qquad &r\in (0,\sqrt{6}],\\
        &\left(1+\frac{1}{2}\frac{1-\frac{1}{2}r^2}{(1+\frac{1}{2}r^2)^2}\right)d, & r\in (\sqrt{6},\infty).
        \end{aligned}
        \right.
    \end{equation*}
    The solution to the equation $\int_0^a \bar{\beta}(r)dr=2/a$ is then $a_0=\sqrt{\frac{32}{15d}}$. The LSI constant $\chlsi$ satisfies
    \begin{align*}
        \chlsi\le a_0^2 \exp\left(\int_0^{a_0} r\bar{\beta}(r)dr-1\right)=\frac{32}{15d}.
    \end{align*}
    According to \cite{vempala2019rapid}, the iteration complexity is of order $\Tilde{O}({d}/{\epsilon})$, where $\tilde{O}$ hides only numerical constants and poly-logarithmic factors. 
    
    \subsection{Example 5} 
We next consider the following potential function given by
\begin{equation*}
        f(x)=\left\{
    \begin{aligned}
         &d(1+\frac{1}{2b})\log |x|-(\frac{d}{4}-1)\log\log|x|+\frac{d}{4}\log(1+2b(\log|x|)^{-1})\\
         &-(\frac{d}{4}-1)\log 2+\frac{d}{4}\log b &~|x|>e \\
        &(d-1)\log |x|-(d-1)\log g_{in}^{-1}(|x|)+\frac{d}{2}g_{in}^{-1}(|x|)^2 +\log g_{in}'(g_{in}^{-1})(|x|) &~ 0\le |x|\le e
    \end{aligned}
    \right.    
\end{equation*}
To study the tail-behavior of the original potential function $f$, we compare it to another potential function $\Tilde{f}(x)=d(1+\frac{1}{2b})\log (1+|x|)-(\frac{d}{4}-1)\log\log(e+|x|)$. According to \cite{wang2014simple}, with $b=\frac{d}{2\vartheta}$, if $d<4$, $\Tilde{f}$ satisfies the weak Poincar\'e inequality with $\vartheta$ being the degree of freedom. If $d=4$, $\Tilde{f}$ satisfies Poincar\'e inequality with $\vartheta$ being the degree of freedom. If $d>4$, $\Tilde{f}$ satisfies the super Poincar\'e inequality with $\vartheta$-degree of freedom.

The transformed potential is given by  
$$
f_h(x)=\frac{d}{2}|x|^2+\frac{d}{4}\log\left(1+\frac{1}{2}|x|^2\right).
$$
The corresponding density function satisfies LSI and log-concavity assumption. The two eigenvalues of the Hessian matrix are:
    \begin{align*}
        \lambda_1=d\left[ 1+\frac{1}{4}\frac{1}{1+\frac{1}{2}|x|^2}\right], \quad\text{and}\quad
        \lambda_2=d\left[1+\frac{1}{4}\frac{1-\frac{1}{2}|x|^2}{(1+\frac{1}{2}|x|^2)^2}\right].
    \end{align*}
    For all $x\in \mb{R}^d$, $0<\lambda_i\le \frac{5}{4}d$ for $i=1,2$. Therefore the transformed potential $f_h$ is gradient Lipschitz with parameter $\frac{3}{2}d$. To find the LSI parameter we use \cite[ Cororllary 1.4]{chen1997estimates}. For all $x\in \mb{R}^d$: $\lambda_2\le \lambda_1$. 
    \begin{equation*}
        \bar{\beta}(r)=\inf_{|x|>r} \lambda_2=\left\{
        \begin{aligned}
        &\frac{31}{32}d, \qquad &r\in (0,\sqrt{6}],\\
        &\left(1+\frac{1}{2}\frac{1-\frac{1}{2}r^2}{(1+\frac{1}{2}r^2)^2}\right)d, & r\in (\sqrt{6},\infty).
        \end{aligned}
        \right.
    \end{equation*}
    The solution to the equation $\int_0^a \bar{\beta}(r)dr=2/a$ is then $a_0=\sqrt{{64}/{31d}}$. The LSI constant $\chlsi$ satisfies
    \begin{align*}
        \chlsi\le a_0^2 \exp\left(\int_0^{a_0} r\bar{\beta}(r)dr-1 \right)=\frac{64}{31d}.
    \end{align*}
    According to \cite{vempala2019rapid}, the iteration complexity is of order $\Tilde{O}({d}/{\epsilon})$, where $\tilde{O}$ hides only numerical constants and poly-logarithmic factors. 
    \subsection{Example 6}
 As the limiting example of the previous three examples, we consider the potential function
 \begin{align*}
     f(x)=\begin{cases}
      d(1+\frac{1}{2b})\log |x|-(\frac{d}{2}-1)\log\log|x|+\log 2+\frac{d}{2}\log b &~\text{for}~|x|>e \\
    (d-1)\log |x|-(d-1)\log g_{in}^{-1}(|x|)+\frac{d}{2}g_{in}^{-1}(|x|)^2 +\log g_{in}'(g_{in}^{-1})(|x|), &~\text{for}~ 0\le |x|\le e
     \end{cases}
 \end{align*}
 We introduce $\Tilde{f}(x)=d(1+\frac{1}{2b})\log (1+|x|)-(\frac{d}{2}-1)\log\log(e+|x|)$ which has similar tail-behavior as the potential $f$ above. According to \cite{wang2014simple}, with $b=\frac{d}{2\vartheta}$, if $d=1$, $\Tilde{f}$ satisfies the weak Poincar\'e inequality with $\vartheta$ being the degree of freedom. If $d=2$, $\Tilde{f}$ satisfies Poincar\'e inequality with $\vartheta$ being the degree of freedom. If $d>2$, $\Tilde{f}$ satisfies the super Poincar\'e inequality with $\vartheta$-degree of freedom and it induces a density function which has heavier tail than the multivariate $t$-distribution with $\vartheta$-degree of freedom.
 
 The transformed potential is $f_h(x)=\frac{d}{2}|x|^2$. The Hessian matrix is $\nabla^2 f_h(x)=d I_d$. Therefore $f_h$ is log-concave with parameter $d$ and the corresponding density satisfies LSI with parameter $\chlsi \le {2}/{d}$.    According to \cite{vempala2019rapid}, the iteration complexity is of order $\Tilde{O}({d}/{\epsilon})$, where $\tilde{O}$ hides only numerical constants and poly-logarithmic factors.

\section{Proofs} 
In this section, we will prove the theorems stated in Sections \ref{sec:THM}-\ref{sec:Examples}.
\subsection{Analysis of the transformation maps}\label{AnalyTran}
In this section we first analyze the transformation map induced by $g$ defined in \eqref{G1}.
\begin{lemma}\label{TRPotential} If the potential function $f$ satisfies \cref{ass:A0}, then we have 
\begin{equation}\label{GradientTHM} \small
    \nabla f_h(x)=\left\{
    \begin{aligned}
    &\left[ f'(g_{in}(|x|)) g_{in}'(|x|)-\frac{g_{in}''(|x|)}{g_{in}'(|x|)}-(d-1)\frac{g_{in}'(|x|)}{g_{in}(|x|)}+\frac{d-1}{|x|} \right]\frac{x}{|x|}  &|x|<b^{-\frac{1}{\beta}},\\
    &\left[ f'(e^{b|x|^\beta})b\beta|x|^{\beta-1}e^{b|x|^\beta}-\beta bd|x|^{\beta-1}+\frac{d-\beta}{|x|} \right]\frac{x}{|x|}  &|x|\ge b^{-\frac{1}{\beta}}.
    \end{aligned}
    \right.
\end{equation}
and $\nabla^2 f_h(x)$ has two eigenvalues $\lambda_1=\lambda_1(|x|)$ and $\lambda_2=\lambda_2(|x|)$ with $\lambda_1, \lambda_2$ defined as
\begin{enumerate}
    \item When $|x|<b^{-\frac{1}{\beta}}$: 
       \begin{align}
       \lambda_1(|x|)&= f''(g_{in}(|x|))(g_{in}'(|x|))^2+f'(g_{in}(|x|))g_{in}''(|x|)-\frac{g_{in}^{(3)}(|x|)}{g_{in}'(|x|)}+(\frac{g_{in}''(|x|)}{g_{in}'(|x|)})^2 \nonumber \\
        &\quad -(d-1)\frac{g_{in}''(|x|)}{g_{in}(|x|)}+(d-1)(\frac{g_{in}'(|x|)}{g_{in}(|x|)})^2-(d-1)|x|^{-2},  \label{EI1}\\
        \lambda_2(|x|)&=f'(g_{in}(|x|))g_{in}'(|x|)|x|^{-1}-\frac{g_{in}''(|x|)}{|x|g_{in}'(|x|)}-(d-1)\frac{g_{in}'(|x|)}{|x|g_{in}(|x|)}+(d-1)|x|^{-2}. \label{EI2}
     \end{align}
     \item When $|x|\ge b^{-\frac{1}{\beta}}$:
     \begin{align} \small
       \lambda_1(|x|)&=f''(e^{b|x|^\beta})b^2\beta^2|x|^{2(\beta-1)}e^{2b|x|^\beta}+f'(e^{b|x|^\beta})(\beta(\beta-1)b|x|^{\beta-2}+\beta^2b^2|x|^{2(\beta-1)})e^{b|x|^\beta}, \nonumber\\
       &\quad -b\beta(\beta-1)d|x|^{\beta-2}-(d-\beta)|x|^{-2}  \label{EO1}\\
       \lambda_2(|x|)&=f'(e^{b|x|^\beta})b\beta|x|^{\beta-1}e^{b|x|^\beta}|x|^{-1}-b\beta d|x|^{\beta-2}+(d-\beta)|x|^{-2}. \label{EO2}
     \end{align}
\end{enumerate}
\end{lemma}
\begin{proof}[Proof of \cref{TRPotential}]\label{PFTRPotential}
For a general transformation map induced by $h$, the transformed potential $f_h$ can be represented as
    \begin{align*}
            f_h(x)&=f(g(|x|))-\log \det (\nabla h(x)) \\
                  &=f(g(|x|))-\log g'(|x|)-(d-1)\log g(|x|)+(d-1)\log |x|.
    \end{align*}
     The gradient of the transformed potential $f_h$ is
    \begin{align}\label{Gradient}
        \nabla f_h(x)&=\left[ f'(g(|x|))g'(|x|)-\frac{g''(|x|)}{g'(|x|)}-(d-1)\frac{g'(|x|)}{g(|x|)}+\frac{d-1}{|x|} \right]\frac{x}{|x|}.
    \end{align}
    Now, \eqref{GradientTHM} follows immediately as a consequence of \eqref{G1} and \eqref{Gradient}. The Hessian matrix of $f_h$ can be represented as
    \begin{align*}
        \nabla^2 f_h(x)&=F_1(|x|)\frac{xx^T}{|x|^2}+F_2(|x|) I_d
    \end{align*}
    with 
    \begin{align*}
         F_1(|x|)&=f''(g(|x|))g'(|x|)^2+f'(g(|x|))g''(|x|)-f'(g(|x|))\frac{g'(|x|)}{|x|}-\frac{g'''(|x|)}{g'(|x|)}\\
          &\ \ \ \ +(\frac{g''(|x|)}{g'(|x|)})^2-(d-1)\frac{g''(|x|)}{g(|x|)}+(d-1)(\frac{g'(|x|)}{g(|x|)})^2\\
          &\ \ \ \ +(d-1)\frac{g'(|x|)}{g(|x|)}\frac{1}{|x|}-\frac{2(d-1)}{|x|^2}+\frac{g''(|x|)}{|x|g'(|x|)},\\
          F_2(|x|)&=\left(f'(g(|x|))g'(|x|)-\frac{g''(|x|)}{g'(|x|)}-(d-1)\frac{g'(|x|)}{g(|x|)}+\frac{d-1}{|x|}\right)|x|^{-1}.
    \end{align*}
    Therefore the two eigenvalues of $\nabla ^2 f_h(x)$, $\lambda_1$ and $\lambda_2$ can be written as 
    \begin{align}
        \lambda_1&=f''(g(|x|))g'(|x|)^2+f'(g(|x|))g''(|x|)-\frac{g'''(|x|)}{g'(|x|)} +(\frac{g''(|x|)}{g'(|x|)})^2 \nonumber \\
         & \quad -(d-1)\frac{g''(|x|)}{g(|x|)}+(d-1)(\frac{g'(|x|)}{g(|x|)})^2-\frac{(d-1)}{|x|^2}, \label{GE1} \\
          \lambda_2&=\left(f'(g(|x|))g'(|x|)-\frac{g''(|x|)}{g'(|x|)}-(d-1)\frac{g'(|x|)}{g(|x|)}+\frac{d-1}{|x|}\right)|x|^{-1} \label{GE2}.
    \end{align}
   The conclusions in \eqref{EI1},\eqref{EI2},\eqref{EO1},\eqref{EO2} can be calculated directly from \eqref{G1}, \eqref{GE1} and \eqref{GE2}. 
\end{proof}
With the above result on the transformation map $h$, we can prove  \cref{PropertyG}. 

\begin{proof}[Proof of \cref{PropertyG}:] We first show that for all $\beta\in [1,2]$, $g\in\mc{C}^3((0,\infty))$. It suffices to show that $g$ is three times continuously differentiable at $r=b^{-1/\beta}$. Based on \eqref{G1}, we have
    \begin{align*}
     &g_{in}({b^{-\frac{1}{\beta}}}_{-})=e=e^{b r^\beta}|_{r=b^{-\frac{1}{\beta}}},\\
     &g_{in}'({b^{-\frac{1}{\beta}}}_{-})=\beta b^{\frac{1}{\beta}}e=(e^{b r^\beta})'|_{r=b^{-\frac{1}{\beta}}},\\
     &g_{in}''({b^{-\frac{1}{\beta}}}_{-})=\beta(2\beta-1)b^{\frac{2}{\beta}}e=(e^{b r^\beta})''|_{r=b^{-\frac{1}{\beta}}},\\
     &g_{in}'''({b^{-\frac{1}{\beta}}}_{-})=\beta(5\beta^2-6\beta+2)b^{\frac{3}{\beta}}e=(e^{b r^\beta})^{(3)}|_{r=b^{-\frac{1}{\beta}}}.
    \end{align*}
Next we show $g$ is monotone increasing. From \cref{ass:G0}, we know that  $g_{in}$ is increasing on the interval $(0,b^{-\frac{1}{\beta}})$. For $r\in [b^{-\frac{1}{\beta}},\infty)$, $g'(r)=b\beta r^{\beta-1}e^{br^\beta}>0$. Combined with the fact that $g\in \mc{C}^3((0,\infty))$, we obtain that $g$ is monotone increasing on the interval $(0,\infty)$. Furthermore, $g(0)=g_{in}(0)=0$ and $\lim_{r\to +\infty}g(r)=\lim_{r\to+\infty} e^{br^\beta}=+\infty$. Therefore $g$ is also onto and invertible. 
\end{proof}
\begin{corollary}\label{EvalueC2} With function $g$ defined in \eqref{G1} and $g_{in}$ satisfying \cref{ass:G0}, if the potential function $f$ satisfies \cref{ass:A0}, then the transformed potential function $f_h$ defined by \eqref{Tdens} is twice continuously differentiable, i.e $f_h\in \mc{C}^2(\mb{R}^d)$.
\end{corollary}
\begin{proof}[Proof of Corollary \ref{EvalueC2}:] Under the assumptions in \cref{EvalueC2}, with the results in \cref{PropertyG} and \eqref{GE1},\eqref{GE2}, we have $f_h\in \mc{C}^2(\mb{R}^d \setminus \{0\})$. Therefore it remains to show that $\lim_{|x|\to 0^+}\lambda_i(|x|)$ are well-defined for $i=1,2$. According to \eqref{Tdens}, we can represent the two eigenvalues of $\nabla^2 f_h(x)$ as for all $r<b^{-\frac{1}{\beta}}$: 
\begin{align*}
    \lambda_1(r)&= f_h''(r)=f''(g_{in}(r))g_{in}'(r)^2+f'(g_{in}(r))g_{in}''(r)-(d-1)(\frac{d^2}{dr^2}\log \frac{g_{in}(r)}{r})+\frac{d^2}{dr^2}\log g_{in}'(r)   \\
    \lambda_2(r)&=\frac{f_h'(r)}{r}=f'(g_{in}(r))\frac{g_{in}'(r)}{r}-\frac{\frac{d}{dr}\log g_{in}'(r)}{r}-(d-1)\frac{\frac{d}{dr}\log \frac{g_{in}(r)}{r}}{r}
\end{align*}
Since $f\in \mc{C}^2(\mb{R})$ and $g$ satisfies \cref{ass:G0}, $\lim_{r\to 0_+}|\lambda_i(r)|<\infty$ for $i=1,2$, which implies that $f_{h}$ is twice continuously differentiable at the origin. Therefore $f_h\in \mc{C}^2(\mb{R}^d)$.
\end{proof}

\subsection{Proof of \cref{PITLMC}}\label{PFTHM1}
We first recall a few definition below. Our proof is based on connections between Lyapunov-based techniques and functional inequality-based techniques for proving ergodicity of diffusion process~\cite{bakry2008rate}.

\begin{definition}[Dissipativity condition]\label{DissC} The Langevin diffusion with drift function $b(x)$ is said to satisfy the dissipativity condition if there exists constants $r,M>0$ such that for all $|x|>M$: $\langle b(x),x \rangle\le -r|x|$.
\end{definition}
\begin{definition}[Lyapunov condition]\label{LyaFun} A function $V\in\mc{D}({\mc{L}})$ with $V\ge 1$ is a Lyapunov function if there exist constants $\lambda, c>0$ and a measurable set $K\subset \mb{R}^d$ such that $\mc{L}V\le \lambda V(-1+c1_{K})$. Equivalently, we say the $\mc{L}$ satisfies the Lyapunov condition. 
\end{definition}
\begin{lemma}\label{DisToLya} Consider the dynamics in~\eqref{TLanD}. If the drift function $-\nabla f_{h}$ satisfies the dissipativity condition with $r>0, M=8(d-1)/r$, then the infinitesimal generator $\mc{L}_h$ of ~\eqref{TLanD} satisfies Lyapunov condition.
\end{lemma}
\begin{proof}[Proof of \cref{DisToLya}]\label{PfDistoLya}
We first construct a Lyapunov function $V$ with respect to the generator $\mc{L}_h$ as
\begin{equation*}
    V(x)=\left\{
    \begin{aligned}
    & 1 \ \ \ \   &|x|\le \frac{M}{2},\\
    & P(|x|)   & \frac{M}{2}<|x|<M, \\
    & e^{a|x|}    & |x|\ge M,
    \end{aligned}
    \right.
\end{equation*}
where $P: [\frac{M}{2},M] \to [1, e^{aM}] $ is a monotone increasing function such that $V\in \mc{C}^2(\mb{R}^d)$ and $V\ge 1$ for all $x\in \mb{R}^d$. When $|x|\ge M$,  we have that
\begin{align*}
    \mc{L}_h V(x)&=-\nabla f_h (x)\cdot \nabla (e^{a|x|})+\Laplace (e^{a|x|}) \\
    &=-\nabla f_h(x)\cdot (ae^{a|x|}\frac{x}{|x|})+ae^{a|x|}\left( \frac{d}{|x|}+a-\frac{1}{|x|} \right)\\
    &\le ae^{a|x|}\left( -r+a+\frac{d-1}{|x|} \right).
\end{align*}
Picking $a=r/2$, we obtain that 
\begin{align*}
    \mc{L}_h V(x)\le  \frac{r}{2}V(x)(-\frac{r}{2}+\frac{d-1}{|x|}), \qquad \qquad \forall |x|\ge M.
\end{align*}
Since $M=8(d-1)/r>4(d-1)/r$, we obtain that $\mc{L}_h V(x)\le -(r^2/8)V(x)$ for all $|x|\ge M$.\\
When $0\le |x|<M$, by the fact that $V\in \mc{C}^2(\mb{R}^d)$ and $V\ge 1$, there exists $A_{r,d}$ such that
\begin{align*}
    \frac{\mc{L}_h V(x)}{V(x)}\le A_{r,d} \qquad\qquad \forall \ 0\le |x|<M,
\end{align*}
where $A_{r,d}=\max_{4(d-1)/r \le |x|\le 8(d-1)/r} \left( -r P'(|x|)+\Laplace (P(|x|))\right)\vee 0$.\\
\par Therefore if $-\nabla f_h$ satisfies dissipativity condition with constant $r>0$, the corresponding generator $\mc{L}_h$ satisfies Lyapunov condition with $\lambda=\lambda_r=r^2/8$, $c=c_{r,d}=A_{r,d}/\lambda_r$ an
\begin{align}\label{eq:temp01}
    K=K_{r,d}=\{x\in \mb{R}^d: 0\le |x|\le 8(d-1)/r \}.
\end{align}
\end{proof}
We further recall additional definitions to proceed.
\begin{definition}[Local Poincar\'{e} inequality]\label{LocPI} The Markov triple $(\mb{R}^d,\mu,\Gamma)$ satisfies a local Poincar\'{e} inequality on a measurable set $K\subset \mb{R}^d$ with $\mu(K)\in (0,\infty)$ if for some constant $C_K$ and every function $\phi\in \mc{D}(\mc{E})$:
\begin{align*}
    \int_K (\phi-m_K)^2 d\mu \le C_K \int_K \Gamma(\phi) d\mu  
\end{align*}
where $m_K=\int_K \phi\  d\mu /\mu(K)$.
\end{definition}
\begin{lemma}\label{LocPITHM} If the original density satisfies \cref{ass:A0}, then the Markov triple $(\mb{R}^d,\pi_h, \Gamma_h)$ satisfies a local Poincar\'{e} inequality on $K_{r,d}$ defined in~\eqref{eq:temp01}.
\end{lemma}
\begin{proof}[Proof of \cref{LocPITHM}]\label{PfLocPITHM}
According to the classical Poincar\'{e} inequality with respect to Lebesgue measure, there is a universal constant $C>0$ such that for all $u\in W^{1,2}(\mb{R}^d)\subset W^{1,2}(K_{r,d})$:
\begin{align*}
    \int_{K_{r,d}} (u(x)-u_{K_{r,d}})^2 dx\le C \frac{d-1}{r} \int_{K_{r,d}} |\nabla u|^2 dx
\end{align*}
where $u_{K_{r,d}}=\int_{K_{r,d}} u(x)\ dx$. To prove the local Poincar\'{e} inequality for $(\mb{R}^d,\pi_h, \Gamma_h)$, without loss of generality, we assume that $\phi\in \mc{D}(\mc{E})$ and $\phi_{K_{r,d}}=\int_{K_{r,d}} \phi(x)\ dx=0$. Then 
\begin{align*}
    \int_{K_{r,d}} (\phi-m_{K_{r,d}})^2 e^{-f_h(x)} dx&=\int_{K_{r,d}} \phi(x)^2 e^{-f_h(x)} dx-(\int_{K_{r,d}} \phi(x)e^{-f_h(x)}dx)^2/|K_{r,d}| \\
    &\le \left( \sup_{x\in K_{r,d}} e^{-f_h(x)} \right)\int_{K_{r,d}} \phi(x)^2 dx\\
    &=  \left( \sup_{x\in K_{r,d}} e^{-f_h(x)} \right)\int_{K_{r,d}} (\phi-\phi_{K_{r,d}})^2 dx\\
    &\le C\frac{d-1}{r}\left ( \sup_{x\in K_{r,d}} e^{-f_h(x)} \right)\int_{K_{r,d}}|\nabla \phi(x)|^2 dx  \\
    &\le C\frac{d-1}{r}\left ( \sup_{x\in K_{r,d}} e^{-f_h(x)} \right)
    \left ( \sup_{x\in K_{r,d}} e^{f_h(x)} \right) \int_{K_{r,d}} \Gamma_h(\phi) e^{-f_h(x)} dx.
\end{align*}
Therefore the Markov triple $(\mb{R}^d,\pi_h, \Gamma_h)$ satisfies a local Poincar\'{e} inequality on $K_{r,d}$ with constant $C_{r,d}=(C(d-1)/r) \left ( \sup_{x\in K_{r,d}} e^{-f_h(x)} \right)
    \left ( \sup_{x\in K_{r,d}} e^{f_h(x)} \right)$.
\end{proof}
\begin{lemma}\label{TranPITHM} If the infinitesimal generator $\mc{L}_h$ satisfies the dissipativity condition with constant $r>0$ and the original density $f$ satisfies \cref{ass:A0}, then the Markov triple $(\mb{R}^d,\pi_h, \Gamma_h)$ also satisfies a Poincar\'{e} inequality. 
\end{lemma}
\begin{proof}[Proof of \cref{TranPITHM}]\label{PfTranPITHM} 
According to \cref{DisToLya}, $\mc{L}_h$ satisfies Lyapunov condition with $\lambda=\lambda_r=r^2/8$, $K$ as in~\eqref{eq:temp01} and $c=c_{r,d}=A_{r,d}/\lambda_r$. Therefore for all $m\in \mb{R}$ and $\phi\in \mc{D}(\mc{E}_h)$:
\begin{align}\label{Eq1}
    \int_{\mb{R}^d} (\phi-m)^2 e^{-f_h(x)}dx &\le \int_{\mb{R}^d} (-\frac{8\mc{L}_h V}{r^2 V}+\frac{8A_{r,d}}{r^2}1_{K_{r,d}})(\phi-m)^2 e^{-f_h(x)} dx \nonumber \\
    &=-\frac{8}{r^2} \int_{\mb{R}^d} \frac{\mc{L}_h V}{V}(\phi-m)^2 e^{-f_h(x)}dx +\frac{8A_{r,d}}{r^2}\int_{K_{r,d}} (\phi-m)^2 e^{-f_h(x)}dx.
\end{align}
Choosing $m=\int_{K_{r,d}}\phi\ e^{-f_h(x)}dx / \int_{K_{r,d}} e^{-f_h(x)} dx$, the second term in \eqref{Eq1} can be bounded as a result of \cref{LocPITHM}:
\begin{align}\label{TERM2BD}
    \frac{8A_{r,d}}{r^2}\int_{K_{r,d}} (\phi-m)^2 e^{-f_h(x)}dx&\le \frac{8A_{r,d}}{r^2} C_{r,d} \int_{K_{r,d}} \Gamma_h(\phi) e^{-f_h(x)}dx.
\end{align}
The first term in \eqref{Eq1} can be bounded by integration by parts, and the diffusion property of $\Gamma_h$ as
\begin{align}\label{TERM1BD}
    -\frac{8}{r^2} \int_{\mb{R}^d} \frac{\mc{L}_h V}{V}(\phi-m)^2 e^{-f_h(x)}dx &= \frac{8}{r^2}  \int_{\mb{R}^d} \Gamma_h(\frac{(\phi-m)^2}{V},V) e^{-f_h(x)}dx   \nonumber \\
    &= \frac{8}{r^2}\int_{\mb{R}^d} \left( \frac{2(\phi-m)}{V}\Gamma_h(\phi-m,V)-\frac{(\phi-m)^2}{V^2}\Gamma_h(V) \right)e^{-f_h(x)}dx \nonumber \\
    &\le \frac{8}{r^2}\int_{\mb{R}^d} \Gamma_h(\phi-m) e^{-f_h(x)}dx \nonumber \\
    &= \frac{8}{r^2}\int_{\mb{R}^d} \Gamma_h(\phi) e^{-f_h(x)}dx .
\end{align}
Combining the results in \eqref{TERM2BD} and \eqref{TERM1BD}, we prove the Markov triple $(\mb{R}^d,\pi_h, \Gamma_h)$ satisfies a Poincar\'{e} inequality with constant $C$ defined as 
\begin{align*}
    C&=\frac{8}{r^2}(1+A_{r,d}C_{r,d})\qquad \text{with} \\
    A_{r,d}&=\max_{4(d-1)/r\le |x|\le 8(d-1)/r} \left( -r P'(|x|)+\Laplace (P(|x|))\right)\vee 0, \\
    C_{r,d}&=\frac{C(d-1)}{r} \left ( \sup_{x\in K_{r,d}} e^{-f_h(x)} \right)
    \left ( \sup_{x\in K_{r,d}} e^{f_h(x)} \right).
\end{align*}
\end{proof}
\begin{proof}[Proof of \cref{PITLMC}]\label{PfPITLMC}
Applying \eqref{GradientTHM} in \cref{TRPotential}, when $|x|\ge 1/b$, we have
\begin{align}
    \langle \nabla f_h(x),x \rangle&=[ f'(e^{b|x|})be^{b|x|}-b-(d-1)\frac{be^{b|x|}}{e^{b|x|}}+\frac{d-1}{|x|} ]|x| \nonumber \\
    &=[ f'(e^{b|x|})be^{b|x|}-bd]|x|+(d-1) \nonumber \\
    &\ge A|x|-B, \label{Inner}
\end{align}
where the last inequality follows from \cref{ass:A3} with $\alpha=1$ and $\beta=1$. We now make the following claim. \\
\par \textbf{Claim:} The infinitesimal generator $\mc{L}_h$ satisfies the dissipativity condition with the constants 
\begin{align}\label{DissiR}
    r&=\frac{-8(d-1)+\sqrt{64(d-1)^2+32AB(d-1)}}{2B}\in (0,A) ,\\
    M&=\frac{8(d-1)}{r} . \nonumber
\end{align}
If the above \textbf{Claim} holds, \cref{PITLMC} follows from \cref{TranPITHM} and Theorem 3 in \cite{vempala2019rapid}. Furthermore, $C_h$ in the statement is given by 
\begin{align*}
    C_h&=\frac{r^2}{8}(1+A_{r,d}C_{r,d})^{-1}\qquad \text{with} \\
    A_{r,d}&=\max_{4(d-1)/r\le |x|\le 8(d-1)/r} \left( -r P'(|x|)+\Laplace (P(|x|))\right)\vee 0, \\
    C_{r,d}&=\frac{C(d-1)}{r^2} \left ( \sup_{x\in K_{r,d}} e^{-f_h(x)} \right)
    \left ( \sup_{x\in K_{r,d}} e^{f_h(x)} \right).
\end{align*}
with $r$ defined in \eqref{DissiR} and $C$ is a universal constant. We now prove the claim
\par \textbf{Proof of the Claim:}  To prove the \textbf{Claim}, it suffices to show for all $|x|\ge \max\{N_1, 8(d-1)/r\}$, we have $\langle \nabla f_h(x),x \rangle\ge r|x|$. Based on \eqref{Inner}, it further suffices to guarantee
\begin{align*}
    A|x|-B\ge r|x| \qquad \text{and} \qquad \frac{8(d-1)}{r}\ge \frac{1}{b}.
\end{align*}
When $r=A/2, 8(d-1)/r$, with $N_1>3B/A$, the above conditions are easily satisfied, therby completing the proof.
\end{proof}
\subsection{Proof of \cref{LSILMC} and \cref{LSITULA}}
\cref{LSILMC} and \cref{LSITULA} are both built on the intermediate result that the transformed measure $\pi_h$ satisfies a LSI. The proof would rely on the following Holley-Stroock theorem.
\begin{theorem}[Holley-Stroock Theorem~\cite{holley1986logarithmic}]\label{HS} Let $\mu\sim LS(C_{\mu})$ and let $\mu_F=Z_F^{-1}e^{-F}\mu$. If $F$ is bounded, then $\mu_F\sim LS(C_{\mu_F})$ and $C_{\mu_F}\le e^{Osc F}C_{\mu}$ where $OscF:=\sup_{x\in \mb{R}^d} F(x)-\inf_{x\in \mb{R}^d} F(x)$.
\end{theorem}
\noindent As an immediate corollary of \cref{HS}, we have $C_{\mu_F}\ge e^{-Osc F}C_{\mu}$. 
\begin{lemma}\label{TranLSI} If the true target density $\pi$ satisfies \cref{ass:A0} and \cref{ass:A1}, then the transformed density $\pi_h$ satisfies a LSI. 
\end{lemma}
\begin{proof}[Proof of \cref{TranLSI}]\label{PFTranLSI} Based on \eqref{EO1} and $\eqref{EO2}$ in \cref{AnalyTran}, when $|x|\ge b^{-\frac{1}{\beta}}$, we have
\begin{align}
    \lambda_1&=f''(e^{b|x|^\beta})b^2\beta^2|x|^{2(\beta-1)}e^{2b|x|^\beta}+f'(e^{b|x|^\beta})(\beta(\beta-1)b|x|^{\beta-2}+\beta^2b^2|x|^{2(\beta-1)})e^{b|x|^\beta} \nonumber\\
       & \quad -\beta(\beta-1)b|x|^{\beta-2}-(d-\beta)|x|^{-2} \nonumber   \\
    &=f''(\psi(|x|))\psi'(|x|)^2+f'(\psi(|x|))\psi''(|x|)-\beta(\beta-1)b|x|^{\beta-2}-(d-\beta)|x|^{-2}\label{GE1Tra1} \\
    \lambda_2&=f'(e^{b|x|^\beta}+C_\beta)b\beta|x|^{\beta-1}e^{b|x|^\beta}|x|^{-1}-b\beta d|x|^{\beta-2}+(d-\beta)|x|^{-2} \nonumber\\
    &=f'(\psi(|x|))\psi'(|x|)|x|^{-1}-b \beta d|x|^{\beta-2}+(d-\beta)|x|^{-2} \label{GE2Tra1}.
\end{align}
where $\psi(r)=e^{br^\beta}$ for all $r\ge b^{-\frac{1}{\beta}}$.
If $f$ satisfies \cref{ass:A1}, then for all $|x|\ge \Tilde{N}_1:= \max\{N_3,b^{-\frac{1}{\beta}}\}$: $\lambda_1(|x|)\ge \rho$ for $i=1,2$.
We can then construct two potentials:
    \begin{equation*}
    \Tilde{f}_h(x)=\left\{
    \begin{aligned}
    & f_h(x) \ \ \ \ \ \ &|x|>\Tilde{N}_1,\\
    & g_h(x) & |x|\le \Tilde{N}_1,\\
    \end{aligned}
    \right.
    \ \ \ \ \ \ \ 
    \Bar{f}_h(x)=\left\{
    \begin{aligned}
    & 0 \ \ \ \ \ \ \ \ \ &|x|>\Tilde{N}_1,\\
    & f_h(x)-g_h(x) & |x|\le \Tilde{N}_1 .
    \end{aligned}
    \right.
    \end{equation*}
where $g_h :\{|x|\le \Tilde{N}_1\}\subset \mb{R}^d \to \mb{R}$ is chosen such that $\Tilde{f}_h \in \mc{C}^2(\mb{R}^d)$ and $\nabla^2 g_h(x)\succeq \rho I_d$ for all $|x|\le \Tilde{N}_1$. Therefore, $\nabla^2 \Tilde{f}_h(x)\succeq \rho I_d$ for all $x\in \mb{R}^d$ i.e $\Tilde{f}_h$ is $\rho$-strongly convex which implies that the measure $\exp(-\Tilde{f_h}(x))dx \sim LS(2/\rho)$(see \cite{bakry1985diffusions}). Meanwhile, $\Bar{f}_h$ is compactly supported on $\{|x|\le \Tilde{N}_1\}$ and $f_h,g_h\in \mc{C}^2(\mb{R}^d)$, which implies that $\Bar{f}_h$ is bounded, i.e $Osc \Bar{f}_h<\infty$ . Last according to the Holley-Stroock theorem and the fact that $\pi_h \propto\exp(-f_h)= \exp(-\Tilde{f}_h)\exp(-\Bar{f}_h)$, 
\begin{align}
    \pi_h\sim LS(\chlsi) \quad\text{with}\quad \chlsi=2e^{Osc \Bar{f_h}}/\rho. \label{LSIConst}
\end{align}
\end{proof}
\begin{proof}[Proof of \cref{LSILMC}]\label{PFLSILMC} 
 The two inequalities in \cref{LSILMC} follows from \cref{TranLSI} and theorem 4 in \cite{vempala2019rapid}. The constant $\chlsi$ in \cref{LSILMC} is the same $\chlsi$ in \eqref{LSIConst}.
\end{proof}
\begin{lemma}\label{GLcond} If the potential function $f$ satisfies \cref{ass:A2}, then the transformed potential $f_{h}(x)$ satisfies the gradient Lipschitz condition, i.e. there exists $L_{h}>0$ such that for all $x,y\in \mb{R}^d$, we have $|\nabla f_{h}(x)-\nabla f_{h}(y)|\le L_{h}|x-y|$.
\end{lemma}
\begin{proof}[Proof of \cref{GLcond}:] It suffices to prove that there is a constant $L_{h}$ such that $\nabla^2 f_{h}(x) \preceq L_{h} I_d$ for all $x\in \mb{R}^d$, i.e $\lambda_1(|x|),\lambda_2(|x|) \le L_{h}$ for all $x\in\mb{R}^d$. Based on \eqref{GE1Tra1},\eqref{GE2Tra1} in the proof of \cref{TranLSI}, and the fact that $f$ satisfies \cref{ass:A2}, we have when $|x|\ge \Tilde{N}_2:= \max\{ N_4, b^{-\frac{1}{\beta}}\}$: $\lambda_i(|x|)\le L$ for $i=1,2$.
When $|x|\le \Tilde{N}_2$, since $f_{h} \in \mc{C}^2(\mb{R}^d)$,
\begin{align*}
    \max_{|x|\le \Tilde{N}_2}\lv \nabla^2 f_{h}(x) \rv<\infty.
\end{align*}
Therefore the transformed density $f_h$ is gradient Lipschitz with parameter $L_{h}$ defined by
\begin{align}\label{GLpara}
    L_{h}=\max\{ L, \max_{|x|\le \Tilde{N}_2} \lv \nabla^2 f_h(x) \rv \}.
\end{align}
\end{proof}
\begin{proof}[Proof of \cref{LSITULA}]\label{PFLSITULA}
From \cref{GLcond} we have that the transformed potential $f_h$ has Lipschitz gradients with parameter $L_h=\max\{ L, \max_{|x|\le \Tilde{N}_2} \lv \nabla^2 f_h(x) \rv \}$. Furthermore, as shown in equation \eqref{LSIConst} in the proof of \cref{TranLSI}, $\pi_h \sim LS(\chlsi)$ with $\chlsi=2e^{Osc \hat{f}_h}/\rho$. Hence, we can apply \cite[Theorem 1]{vempala2019rapid} to obtain that when $0<\gamma<\frac{1}{2\chlsi L_h^2}$, 
\begin{align}\label{TULAconv}
    H_{\pi_h}(\rho_n)\le e^{-\frac{2\gamma n}{\chlsi}}H_{\p_h}(\rho_0)+4\chlsi L_h^2\gamma d.
\end{align}
Now, applying \cref{Phithm} with $\Phi(x)=x\log x$ to \eqref{TULAconv}, we get
\begin{align}\label{Sampleconv}
    H_{\pi}(\nu_n)\le e^{-\frac{2\gamma n}{\chlsi}}H_{\pi}(\nu_0)+4\chlsi L_h^2\gamma d.
\end{align}
The mixing time estimate in the theorem instantly follows from equation \eqref{Sampleconv}.
\end{proof}


\subsection{Proof of \cref{mLSILMC} and \cref{thm:mlsi and tail ass}}

In this section we will prove \cref{mLSILMC}
and \cref{thm:mlsi and tail ass}. First we introduce a result which explains the relation between \cref{ass:A5} and \cref{ass:Degenerate convexity at infinity}.
\begin{lemma}\label{5toDGCon} If a potential function $f$ satisfies \cref{ass:A5}, then the transformed potential $f_{h}$ satisfies \cref{ass:Degenerate convexity at infinity}.
\end{lemma}
\begin{proof}[Proof of \cref{5toDGCon}:] If a potential function $f$ satisfies \cref{ass:A5} with parameters $\mu,N_2$ and $\theta$, then when $|x|\ge b^{-\frac{1}{\beta}}$, the eigenvalues of $\nabla^2 f_h(x)$ are studied in \eqref{EO1} and \eqref{EO2}. Applying $\psi(|x|)=e^{b|x|^\beta}$, we have the following estimates on the eigenvalues: for all $|x|\ge \Tilde{N}_5:= \max\{ b^{-\frac{1}{\beta}}, N_2\}$ we have
\begin{align}
    \lambda_1&=f''(\psi(|x|))\psi'(|x|)^2+f'(\psi(|x|))\psi''(|x|)-b \beta(\beta-1)d|x|^{\beta-2}-(d-\beta)|x|^{-2} \nonumber\\
    &\ge \frac{\mu}{(1+\frac{1}{4}|x|^2)^{\frac{\theta}{2}}},  \nonumber \\
    \lambda_2&=f'(\psi(|x|))\psi'(|x|)|x|^{-1}-b \beta d|x|^{\beta-2} +(d-\beta)|x|^{-2} \nonumber \\
    &\ge  \frac{\mu}{(1+\frac{1}{4}|x|^2)^{\frac{\theta}{2}}}. \nonumber
\end{align}
 where the inequality follows from \cref{ass:A5}. 
 Therefore for all $|x|\ge \Tilde{N}_5$, we have that
 \begin{align*}
     \nabla^2 f_h(x) \succeq \frac{\mu}{(1+\frac{1}{4}|x|^2)^{\frac{\theta}{2}}} I_d.
 \end{align*}
 Meanwhile since $f_h\in C^2(\mb{R}^d)$, we can construct $\Tilde{f}_h\in C^2(\mb{R}^d)$ such that $\Tilde{f}_h(x)=f_h(x)$ for all $|x|\ge \Tilde{N}_5$, $\nabla^2\Tilde{f}_h(x)\succeq \frac{\mu}{(1+\frac{1}{4}|x|^2)^{\frac{\theta}{2}}} I_d$ for all $x\in \mb{R}^d$. Furthermore, since both $f_h$ and $\Tilde{f}_h$ are continuous, 
 \begin{align}\label{xi}
     \xi:=\lv f_h-\Tilde{f_h} \rv_\infty=\max_{|x|\le \Tilde{N}_5} |f_h(x)-\Tilde{f_h}(x)|<\infty
 \end{align}
 Therefore $f_h$ satisfies \cref{ass:Degenerate convexity at infinity} with parameters $\xi_h=\xi$, $\mu_h=\mu$ and $\theta_h=\theta$.
\end{proof}
Next we introduce a result which explains the relation between \cref{ass:A3} and \cref{alphadis}.
\begin{lemma}\label{alphadissi} If a potential function $f$ satisfies \cref{ass:A3} then the transformed potential $f_{h}$ satisfies \cref{alphadis}.
\end{lemma}
\begin{proof}[Proof of \cref{alphadissi}:] If a potential function $f$ satisfies \cref{ass:A3} with parameters $\alpha,A,B$, as we have shown in \eqref{GradientTHM}, for all $|x|\ge b^{-\frac{1}{\beta}}$ we have that
\begin{align}
    \langle \nabla f_{h}(x),x \rangle&=f'(e^{b|x|^\beta})b\beta|x|^{\beta}e^{b|x|^\beta}-\beta bd|x|^{\beta}+(d-\beta) \nonumber \\
    &\ge A|x|^\alpha-B . \nonumber
\end{align}
where the first inequality follows from \cref{ass:A3}. Since $f_{h}\in \mc{C}^2(\mb{R}^d)$, we have
\begin{align*}
    \min_{|x|\le b^{-\frac{1}{\beta}}}\langle \nabla f_{h}(x),x \rangle >-\infty
\end{align*}
Therefore $f_{h}$ satisfies \cref{alphadis} with parameters
\begin{align}\label{paraAB}
    \alpha_h=\alpha, \quad A_{h}=A, \quad B_{h}=\max \{ 0, B,  -\min_{|x|\le b^{-\frac{1}{\beta}}}\langle \nabla f_h(x),x \rangle  \}\in [0,\infty).
\end{align}
\end{proof}
With the above two results, we are ready to prove \cref{mLSILMC}.

\begin{proof}[Proof of \cref{mLSILMC}:] 
Taking the derivative of the KL-divergence from $\rho_t$ to $\pi_h$, we have
\begin{align}\label{dKL}
    \frac{d}{dt} H_{\pi_h}(\rho_t) &= \frac{d}{dt} \int_{\mb{R}^d} \log \left(\frac{\rho_t(x)}{\pi_h(x)}\right)\rho_t(x) dx  \nonumber\\
    &= \int_{\mb{R}^d} \frac{\partial \rho_t(x)}{\partial t}\log\left(\frac{\rho_t(x)}{\pi_h(x)}\right)dx+\int_{\mb{R}^d} \rho_t(x) \frac{\pi_h(x)}{\rho_t(x)} \frac{1}{\pi_h(x)}\frac{\partial \rho_t(x)}{\partial t} dx \nonumber\\
    &=\int_{\mb{R}^d} \nabla \cdot \left( \rho_t(x) \nabla \log \frac{\rho_t(x)}{\pi_h(x)} \right) \log \frac{\rho_t(x)}{\pi_h(x)} dx+0 \nonumber \\
    &=-\int_{\mb{R}^d} \rho_t(x) \left|\nabla \log \frac{\rho_t(x)}{\pi_h(x)}\right|^2 dx=-I_{\pi_h}(\rho_t),
\end{align}
where third identity follows from the Fokker-Planck equation 
$$
\frac{\partial \rho_t}{\partial t}=\nabla \cdot (\rho_t \nabla f_h)+\Delta \rho_t=\nabla \cdot \left( \rho_t \nabla \log \frac{\rho_t}{\pi_h} \right),
$$ 
and the fact that $\int \frac{\partial \rho_t}{\partial t}dx=\frac{d}{dt}\int \rho_t dx=0$. According to \cref{5toDGCon}, we have that $\pi_h$ satisfies \cref{ass:Degenerate convexity at infinity}. Hence, according to \cite[Theorem 1]{erdogdu2020convergence}, $\pi_h$ satisfies a modified LSI, i.e. for all probability densities $\rho$:
\begin{align*}
    H_{\pi_h}(\rho) \le \chmlsi I_{\pi_h}(\rho)^{1-\delta} M_s(\rho+\pi_h)^\delta,
\end{align*}
where $M_s(\rho)=\int_{\mb{R}^d} (1+|x|^2)^{s/2}\rho(x)dx $ is the $s$-th moment of any function $\rho$ and with $\xi$ defined in \eqref{xi}, $\delta$ and $\lambda$ are defined as 
\begin{align}\label{mLSIconst}
    \delta:=\frac{\theta}{s-2+2\theta} \in [0,\frac{1}{2}), \qquad \chmlsi=4e^{2\xi}\mu^{-\frac{s-2}{s-2+2\theta}}. 
\end{align}
Hence \eqref{dKL} can be further written as 
\begin{align}\label{UBdKL}
    \frac{d}{dt} H_{\pi_h}(\rho_t) \le - \lambda^{-\frac{1}{1-\delta}} H_{\pi_h}(\rho_t)^{\frac{1}{1-\delta}}M_s(\rho_t+\pi_h)^{-\frac{\delta}{1-\delta}}.
\end{align}
By \cref{alphadissi}, the transformed potential $f_h$ satisfies \cref{alphadis} with parameters $\alpha_h,A_h,B_h$. Hence, according to \cite[Proposition 2]{toscani2000trend}, under the $\alpha_h$-dissipativity of $f_h$, for all $s\ge 2$:
\begin{align*}
    M_s(\rho_t+\pi_h)\le M_s(\rho_0+\pi_h)+C_st,
\end{align*}
where 
\begin{align}\label{UBCs}
    C_s=\sup_{x\ge 0} \left((ds+s(s-2)-sA_h+sB_h)x^{\frac{s-2}{s-2+\alpha}}-A_h x \right)<\infty.
\end{align}
Therefore the upper bound in \eqref{UBdKL} can improved as 
\begin{align}\label{UB2dKL}
    \frac{d}{dt}H_{\pi_h}(\rho_t) \le -\lambda^{-\frac{1}{1-\delta}}H_{\pi_h}(\rho_t)^{\frac{1}{1-\delta}}\left(M_s(\rho_0+\pi_h)+C_s t \right)^{-\frac{\delta}{1-\delta}}.
\end{align}
Rewriting \eqref{UB2dKL} as
\begin{align}
    -H_{\pi_h}(\rho_t)^{-\frac{1}{1-\delta}}\frac{d}{dt}H_{\pi_h}(\rho_t) \ge (\lambda C_s^\delta)^{-\frac{1}{1-\delta}}(M_s(\rho_0+\pi_h)C_s^{-1}+t)^{-\frac{\delta}{1-\delta}},
\end{align}
and applying Gronwall's inequality, we obtain
\begin{align}\label{UBTLD}
    H_{\pi_h}(\rho_t) &\le (\frac{1-2\delta}{\delta})^{\frac{1-\delta}{\delta}}(\lambda C_s^\delta)^{\frac{1}{\delta}}(M_s(\rho_0+\pi_h)C_s^{-1}+t)^{-\frac{1-2\delta}{\delta}}\le \frac{C}{t^l}.
\end{align}
with $C=(\frac{1-2\delta}{\delta})^{\frac{1-\delta}{\delta}}(\lambda C_s^\delta)^{\frac{1}{\delta}}$ and $l=(1-2\delta)/\delta$.
\end{proof}

To prove~\cref{thm:mlsi and tail ass}, we require the following result on the relationship between \cref{ass:tail assumption on original potential} and \cref{ass:tail assumption for transformed density}.

\begin{lemma}\label{lem: tail ass to tail ass} If the density $\pi$ satisfies \cref{ass:tail assumption on original potential}, then $\pi$ satisfies \cref{ass:tail assumption for transformed density}.
\end{lemma}
\begin{proof}[Proof of \cref{lem: tail ass to tail ass}]\label{pf:tail ass to tail ass} Without loss of generality, we can assume that $N_5\ge e$. When $\lambda\ge N_5\ge e$, 
\begin{align*}
     \pi\left\{ |\cdot|\ge m+\lambda \right\}\le 2\exp\left( -\left( \frac{g^{-1}(\lambda)}{\ctail} \right)^{\alpha_1} \right)
\end{align*}
with $\ctail=\ctail^*$. When $\lambda\in [0,N_5]$, we have
\begin{align*}
    \pi\left\{ |\cdot|\ge m+\lambda \right\}&\le \pi\left\{ |\cdot|\ge m \right\}\\
    &\le 2\exp\left( -\left( \frac{g^{-1}(\lambda)}{\ctail} \right)^{\alpha_1} \right)
\end{align*}
with $\ctail=g^{-1}(N_5)\left( \log \frac{2}{\pi\left\{ |\cdot|\ge m \right\}} \right)^{\frac{1}{\alpha_1}}$. Therefore for all $\lambda\ge 0$,
\begin{align}\label{eq:tail ass original density}
    \pi\left\{ |\cdot|\ge m+\lambda \right\}\le 2\exp\left( -\left( \frac{g^{-1}(\lambda)}{\ctail} \right)^{\alpha_1} \right),
\end{align}
with
\begin{align*}
    \ctail=\max \left\{ \ctail^*, g^{-1}(N_5)\left( \log \frac{2}{\pi\left\{ |\cdot|\ge m \right\}} \right)^{\frac{1}{\alpha_1}} \right\} .
\end{align*}
From \eqref{eq:tail ass original density}, let $X\in \mb{R}^d$ be a random variable with density $\pi$ and $Y:=h^{-1}(X)$. Then $Y\in \mb{R}^d$ is a random variable with density $\pi_h$. We get 
\begin{align*}
    \pi_h\left\{ |\cdot| \ge m_h+ \lambda \right\}&=\mb{P}\left( |Y|\ge m_h+\lambda \right)\\
    &=\mb{P}\left( g^{-1}(|X|)\ge m_h+\lambda \right)\\
    &=\mb{P}\left( |X|\ge g(m_h+\lambda) \right).
\end{align*}
For any fixed $\lambda\ge 0$, we can choose $m_h(\lambda)=g^{-1}(m+g(\lambda))-\lambda$ and we get
\begin{align*}
    \pi_h\left\{ |\cdot| \ge m_h(\lambda)+ \lambda \right\}&\le \mb{P}\left( |X|\ge m+g(\lambda) \right) \\
    &=\pi\left\{ |\cdot|\ge m+g(\lambda) \right\}\\
    &\le 2\exp\left( -\left( \frac{\psi^{-1}(g(\lambda))}{\ctail} \right)^{\alpha_1} \right)\\
    &=2\exp\left( -\left( \frac{\lambda}{\ctail} \right)^{\alpha_1} \right).
\end{align*}
We next claim that there exists a constant $m_h$ such that $m_h(\lambda)\le m_h$ for all $\lambda\ge 0$. To prove the claim, we apply Taylor expansion in the definition of $m_h(\lambda)$ and we get for any $\lambda\ge 0$, there exists a constant $\theta(\lambda)\in [0,m]$ such that
\begin{align*}
    m_h(\lambda)&=g^{-1}\left(g(\lambda)\right)+(g^{-1})'(g(\lambda))m-\lambda\\
    &\le  m \sup_{r\in [0,\infty)} \left(g^{-1})'(r)\right).
\end{align*}
According to our construction of $g$, we have that $\sup_{r\in [0,\infty)} \left(g^{-1})'(r)\right)<\infty$. Therefore we can pick $m_h=m\left(\sup_{r\in [0,\infty)} \left(g^{-1})'(r)\right)\right)$ which is a constant independent of $\lambda$, which proves the claim. Hence, we get for all $\lambda\ge 0$,
\begin{align*}
    \pi_h\left\{ |\cdot| \ge m_h+ \lambda \right\}&\le 2\exp\left( -\left( \frac{\psi^{-1}(g(\lambda))}{\ctail} \right)^{\alpha_1} \right).
\end{align*}
That is, the transformed density $\pi_h$ satisfies \cref{ass:tail assumption for transformed density} with 
\begin{align}
    &\alpha_h=\alpha_1, \label{eq:tail ass para alpha}\\
    &\chtail=\max \left\{ \ctail^*, g^{-1}(N_6)\left( \log \frac{2}{\pi\left\{ |\cdot|\ge m \right\}} \right)^{\frac{1}{\alpha_1}} \right\} , \label{eq:tail ass para ctail}\\ 
    & m_h=m\left(\sup_{r\in [0,\infty)} \left(g^{-1})'(r)\right)\right) \label{eq:tail ass para m}.
\end{align}
\end{proof}
\end{myassump}
\begin{proof}[Proof of \cref{thm:mlsi and tail ass}]\label{pf:mlsi and tail ass} 
Let $\hat{\pi}_h$ be a modified density to $\pi_h$. It's defined as, for $\hat{\gamma},\hat{R}>0$,
\begin{align}\label{eq:modified density}
    \hat{\pi}_h\propto \exp(-\hat{f_h}),\qquad \hat{f_h}(x):=f_h(x)+\frac{\hat{\gamma}}{2}(|x|-\hat{R})_{+}^2.
\end{align}
Here $(|x|-\hat{R})_{+}^2$ is interpreted as $\max\left\{ |x|-\hat{R},0 \right\}^2$. Furthermore, $\hat{R}$ is chosen so that $\pi_h(B(0,\hat{R}))\ge \frac{1}{2}$, where $B(0,\hat{R})$ is an Euclidean ball of radius $\hat{R}$ centered at zero. With this definition, the proof follows immediately from \cref{5toDGCon}, \cite[Theorem 1]{erdogdu2020convergence} and \cite[Theorem 8]{chewi2021analysis}.
\end{proof}

\subsection{Proofs for~\cref{sec:htfunctionalineq}}

\begin{proof}[Proof of \cref{Prop:A1toPoincare}:]
When $|x|\ge g( b^{-\frac{1}{\beta}})=e$, the inverse of $g$ can be represented as
\begin{align*}
    g^{-1}(|x|)=b^{-\frac{1}{\beta}}\log^{\frac{1}{\beta}}|x|.
\end{align*}
Therefore, \cref{ass:A1} can be reformulated as for all $|x|\ge g^{-1}(N_3 \vee b^{-\frac{1}{\beta}}):=g^{-1}(\Tilde{N}_1)$:
\begin{align}
    &b^{\frac{2}{\beta}}[f'(|x|)\beta \log^{1-\frac{2}{\beta}}(|x|)|x|-\beta d \log^{1-\frac{2}{\beta}}(|x|)+(d-\beta)\log ^{-\frac{2}{\beta}}(|x|)]>\rho, \label{WPIA11}\\
    &b^{\frac{2}{\beta}}[f''(|x|)\beta^2\log^{2-\frac{2}{\beta}}(|x|)|x|^2+ f'(|x|)\beta(\beta-1)\log^{1-\frac{2}{\beta}}(|x|)|x| \nonumber \\
    &+f'(|x|)\beta^2\log ^{2-\frac{2}{\beta}}(|x|)|x|-\beta(\beta-1)\log ^{1-\frac{2}{\beta}}(|x|)-(d-\beta)\log^{-\frac{2}{\beta}}(|x|) ]>\rho. \label{WPIA22}
\end{align}
Now, \eqref{WPIA11} gives a lower bound on $f'(|x|)$ of the form:
\begin{align*}
    f'(|x|)\ge \rho b^{-\frac{2}{\beta}}\beta^{-1}\log^{-(1-\frac{2}{\beta})}(|x|)|x|^{-1}+d|x|^{-1}-\frac{d-\beta}{\beta}\log^{-1}(|x|)|x|^{-1}.
\end{align*}
Defining $N:=g^{-1}(\Tilde{N}_1)$ and integrating from $N$ to a larger value with respect to $|x|$, we obtain
\begin{align*}
    f(|x|)&\ge f(N)+\frac{1}{2}\rho b^{-\frac{2}{\beta}}(\log^{\frac{2}{\beta}}(|x|)-\log^{\frac{2}{\beta}}N)+d(\log |x|-\log N)-\frac{d-\beta}{\beta}(\log\log |x|-\log\log N) \\
    &:=C_{N,d}+ \frac{1}{2}\rho b^{-\frac{2}{\beta}}\log^{\frac{2}{\beta}}(|x|)+d\log|x|-\frac{d-\beta}{\beta}\log\log|x|.
\end{align*}
Therefore we have for all $|x|>N$:
\begin{align}\label{A21ExpLowB}
    e^{f(x)}\ge C_{N,d} |x|^{\frac{1}{2}\rho b^{-\frac{2}{\beta}}\log^{\frac{2}{\beta}-1}(|x|)+d} \log^{-\frac{d-\beta}{\beta}}(|x|).
\end{align}
To prove that $\pi$ satisfies the Poincar\'e-type inequalities, we leverage the results in \cite{wang2014simple} and \cite{wang2015functional}. We consider the following quantity with $\vartheta \in (0,2)$ and $x\neq y$:
\begin{align*}
    \frac{e^{f(x)}+e^{f(y)}}{|x-y|^{d+\vartheta}}
\end{align*}
Since $|x-y|^{d+\vartheta}\le 2^{d+\vartheta-1}(|x|^{d+\vartheta}+|y|^{d+\vartheta})$, we have for all $x\neq y$ and $|x|,|y|>N$, we have
\begin{align}\label{Lowerbound1}
     \frac{e^{f(x)}+e^{f(y)}}{|x-y|^{d+\vartheta}}\ge \frac{C_{N,d}}{2^{d+\vartheta-1}}\frac{|x|^{\frac{1}{2}\rho b^{-\frac{2}{\beta}}\log^{\frac{2}{\beta}-1}(|x|)+d}\log^{-\frac{d-\beta}{\beta}}(|x|)+|y|^{\frac{1}{2}\rho b^{-\frac{2}{\beta}}\log^{\frac{2}{\beta}-1}(|y|)+d}\log^{-\frac{d-\beta}{\beta}}(|y|)}{|x|^{d+\vartheta}+|y|^{d+\vartheta}}.
\end{align}
Then, \eqref{WPIA22} gives the lower bound
\begin{align*}
    &f''(|x|)+f'(|x|)\left( \frac{\beta-1}{\beta}\log^{-1}(|x|)|x|^{-1}+|x|^{-1}  \right)\\
    &\quad \ge \rho b^{-\frac{2}{\beta}}\beta^{-2}\log^{-(2-\frac{2}{\beta})}(|x|)|x|^{-2} +\frac{\beta-1}{\beta} \log^{-1}(|x|)|x|^{-2}+\frac{d-\beta}{\beta^2}\log^{-2}(|x|)|x|^{-2}.
\end{align*}
By multiplying $\log^{1-\frac{1}{\beta}}(|x|)|x|$ on both sides, for all $|x|>N$ we obtain
\begin{align*} 
    \frac{d}{d|x|}\left( f'(|x|)\log^{1-\frac{1}{\beta}}(|x|)|x| \right)\ge & \rho b^{-\frac{2}{\beta}}\beta^{-2}\log^{-(1-\frac{1}{\beta})}(|x|)|x|^{-1}+\frac{\beta-1}{\beta} \log^{-\frac{1}{\beta}}(|x|)|x|^{-1}\\
    &+\frac{d-\beta}{\beta}\log^{-1-\frac{1}{\beta}}(|x|)|x|^{-1},
\end{align*}
which implies that
\begin{align*}
    f'(|x|)\ge C_{N,d,1}\log ^{-(1-\frac{1}{\beta})}(|x|)|x|^{-1}+\rho b^{-\frac{2}{\beta}}\beta^{-1}\log^{-(1-\frac{2}{\beta})}(|x|)|x|^{-1}+|x|^{-1}-(d-\beta)\log^{-1}(|x|)|x|^{-1}.
\end{align*}
Further integration implies that for all $|x|>N$, we have
\begin{align}\label{A22ExpLowB}
    e^{f(|x|)}\ge C_{N,d,2}  |x|^{1+\frac{1}{2}\rho b^{-\frac{2}{\beta}}\log^{\frac{2}{\beta}-1}(|x|)+C_{N,d,1}\beta \log^{\frac{1}{\beta}-1}(|x|)}\log^{-(d-\beta)}|x|.
\end{align}
Since $d\ge 1$, \eqref{A21ExpLowB} is stronger than \eqref{A22ExpLowB}, when we apply results in \cite{wang2014simple}, it's enough for us to consider only \eqref{A21ExpLowB}. Therefore we have the following results:
\begin{itemize}
    \item [(1)] When $\beta\in (1,2)$ or $\beta=2,\vartheta<\frac{1}{2}\rho b^{-1}$, we can see that for all $|x|>N$:
    \begin{align*}
        \frac{1}{2}\rho b^{-\frac{2}{\beta}}\log^{\frac{2}{\beta}-1}(|x|)+d-(d+\vartheta)>0
    \end{align*}
    Therefore, with \eqref{Lowerbound1}, conditions in \cite[Theorem 1.1-(3)]{wang2014simple} is satisfied with $$\omega(x)=\frac{C_{N,d}}{2^{d+\vartheta}}|x|^{\frac{1}{2}\rho b^{-\frac{2}{\beta}}\log^{\frac{2}{\beta}-1}(|x|)-\vartheta}\log^{-\frac{d-\beta}{\beta}}(|x|),$$ such that $\lim_{|x|\to \infty}\omega(x)=\infty$. Hence, $\pi\propto \exp(-f)$ satisfies the super-Poincar\'{e} inequality.
    \item [(2)] When $\beta=2, \vartheta=\frac{1}{2}\rho b^{-1}, d=1,2$, we can see that for all $|x|>N$, we have
    \begin{align*}
        \frac{1}{2}\rho b^{-\frac{2}{\beta}}\log^{\frac{2}{\beta}-1}(|x|)+d-(d+\vartheta)=0
    \end{align*}
    Therefore with \eqref{Lowerbound1}, since $d=1,2$ we obtain
    \begin{align*}
        \frac{e^{f(x)}+e^{f(y)}}{|x-y|^{d+\vartheta}}\ge \frac{C_{N,d}}{2^{d+\vartheta-1}}.
    \end{align*}
    Hence, according to in \cite[Theorem 1.1-(1)]{wang2014simple}, $\pi\propto \exp(-f)$ satisfies the Poincar\'e inequality.
    \item [(3)]When $\beta=2, \vartheta=\frac{1}{2}\rho b^{-1}, d\ge 3$, for all $|x|>N$, we have that
    \begin{align*}
        \frac{1}{2}\rho b^{-\frac{2}{\beta}}\log^{\frac{2}{\beta}-1}(|x|)+d-(d+\vartheta)=0
    \end{align*}
    However, the lower bound in \eqref{Lowerbound1} goes to zero as $|x|,|y|\to \infty$. Neither Poincar\'e inequality nor super Poincar\'e inequality is guaranteed. However, according to \cite[Theorem 1.1-(2)]{wang2014simple}, $\pi\propto \exp(-f)$ satisfies the weak Poincar\'e inequality with $\alpha(r)$ as defined in~\eqref{eq:alphar}. 
    When $\beta=2, \vartheta>\frac{1}{2}\rho b^{-1}$, we can see that for all $|x|>N$: 
    \begin{align*}
        \frac{1}{2}\rho b^{-\frac{2}{\beta}}\log^{\frac{2}{\beta}-1}(|x|)+d-(d+\vartheta)<0
    \end{align*}
    Neither Poincar\'e inequality nor super Poincar\'e inequality is guaranteed. However, according to \cite[Theorem 1.1-(2)]{wang2014simple}, $\pi\propto \exp(-f)$ satisfies the weak Poincar\'e inequality with $\alpha(r)$ as in~\eqref{eq:alphar}. 
\end{itemize}
\end{proof}
\begin{proof}[Proof of \cref{Prop:A3toPoincare}:]
Similar to \cref{ass:A1}, \cref{ass:A3} are sufficient conditions for Poincar\'e type inequalities as well. First note that \cref{ass:A3} is equivalent to the following inequality: for all $|x|>N:=g^{-1}(N_1\vee b^{-\frac{1}{\beta}})$ we have
\begin{align*}
    f'(|x|)\ge A\beta^{-1}b^{-\frac{\alpha}{\beta}}\log^{\frac{\alpha}{\beta}-1}(|x|)|x|^{-1}+d|x|^{-1}-\frac{B}{\beta}\log^{-1}(|x|)|x|^{-1}.
\end{align*}
Integrating with respect to $|x|$, we obtain
\begin{align*}
    f(|x|)\ge C_{N,d}+Ab^{-\frac{\alpha}{\beta}}\alpha^{-1}\log^{\frac{\alpha}{\beta}}(|x|)+d\log|x|-\frac{B}{\beta}\log\log|x|.
\end{align*}
For all $|x|\ge N$, we then have
\begin{align*}
    e^{f(|x|)}\ge C_{N,d} |x|^{A\alpha^{-1}b^{-\frac{\alpha}{\beta}}\log^{\frac{\alpha}{\beta}-1}(|x|)+d}\log^{-\frac{B}{\beta}}|x|.
\end{align*}
and
\begin{align}\label{Lowerbound2}
     \frac{e^{f(x)}+e^{f(y)}}{|x-y|^{d+\vartheta}}\ge \frac{C_{N,d}}{2^{d+\vartheta-1}}\frac{|x|^{A\alpha^{-1}b^{-\frac{\alpha}{\beta}}\log^{\frac{\alpha}{\beta}-1}(|x|)+d}\log^{-\frac{B}{\beta}}|x|+|y|^{A\alpha^{-1}b^{-\frac{\alpha}{\beta}}\log^{\frac{\alpha}{\beta}-1}(|y|)+d}\log^{-\frac{B}{\beta}}|y|}{|x|^{d+\vartheta}+|y|^{d+\vartheta}}.
\end{align}
We now consider different cases. 
\begin{itemize}
    \item [(1)] When $\alpha>\beta$ or $\alpha=\beta, \vartheta<A\beta^{-1}b^{-1}$, we can see that for all $|x|>N$ we have that
    \begin{align*}
       A\alpha^{-1}b^{-\frac{\alpha}{\beta}}\log^{\frac{\alpha}{\beta}-1}(|x|)+d-(d+\vartheta)>0.
    \end{align*}
    Therefore, with \eqref{Lowerbound2}, the conditions in \cite[Theorem 1.1-(3)]{wang2014simple} are satisfied with $$\omega(x)=\frac{C_{N,d}}{2^{d+\vartheta}}|x|^{A\alpha^{-1}b^{-\frac{\alpha}{\beta}}\log^{\frac{\alpha}{\beta}-1}(|x|)-\vartheta}\log^{-\frac{B}{\beta}}(|x|),$$
    such that $\lim_{|x|\to \infty}\omega(x)=\infty$. Hence, $\pi\propto \exp(-f)$ satisfies the super Poincar\'e inequality.
    \item [(2)] When $\alpha=\beta, \vartheta=A\beta^{-1}b^{-1}$, we can see that for all $|x|>N$ we have
    \begin{align*}
        A\alpha^{-1}b^{-\frac{\alpha}{\beta}}\log^{\frac{\alpha}{\beta}-1}(|x|)+d-(d+\vartheta)=0.
    \end{align*}
    However, the lower bound in \eqref{Lowerbound2} goes to zero as $|x|,|y|\to \infty$. Hence, Neither the Poincar\'e inequality nor the super Poincar\'e inequality is satisfied. However, according to \cite[Theorem 1.1-(2)]{wang2014simple}, the density $\pi\propto \exp(-f)$ satisfies the weak Poincar\'e inequality with $\alpha(r)$ as in~\eqref{eq:alphar}.
    \item [(3)] When $\alpha=\beta, \vartheta>A\beta^{-1}b^{-1}$, we can see that for all $|x|>N$ we have
    \begin{align*}
         A\alpha^{-1}b^{-\frac{\alpha}{\beta}}\log^{\frac{\alpha}{\beta}-1}(|x|)+d-(d+\vartheta)<0.
    \end{align*}
    Hence, neither the Poincar\'e inequality nor super Poincar\'e inequality is guaranteed. However, according to \cite[Theorem 1.1-(2)]{wang2014simple}, $\pi\propto \exp(-f)$ satisfies the weak Poincar\'e inequality with $\alpha(r)$ as in~\eqref{eq:alphar}.
\end{itemize}
\end{proof}
\begin{proof}[Proof of \cref{Prop:A5toPoincare}:]
Similar as in the proof of \cref{Prop:A1toPoincare}, \cref{ass:A5} is equivalent to the following two inequalities: for all $|x|\ge N$,
\begin{align}
    &f'(|x|)\ge \mu b^{-\frac{2}{\beta}}\beta^{-1} \log^{-(1-\frac{2}{\beta})}(|x|)(1+\frac{1}{4}b^{-\frac{2}{\beta}}\log^{\frac{2}{\beta}}(|x|))^{-\frac{\theta}{2}}|x|^{-1}+d|x|^{-1}-\frac{d-\beta}{\beta}\log^{-1}(|x|)|x|^{-1}, \label{A51WPI}\\
    &f''(|x|)+f'(|x|)\left( \frac{\beta-1}{\beta}\log^{-1}(|x|)|x|^{-1}+|x|^{-1}  \right) \nonumber \\
    &\ge \mu b^{-\frac{2}{\beta}}\beta^{-2}\log^{-(2-\frac{2}{\beta})}(|x|)(1+\frac{1}{4}b^{-\frac{2}{\beta}}\log^{\frac{2}{\beta}}(|x|))^{-\frac{\theta}{2}}|x|^{-2} +\frac{\beta-1}{\beta} \log^{-1}(|x|)|x|^{-2}+\frac{d-\beta}{\beta^2}\log^{-2}(|x|)|x|^{-2}. \label{A52WPI}
\end{align}
Choosing $N'>N$ such that for $|x|>N'$, $b^{-\frac{2}{\beta}}\log^{\frac{2}{\beta}}(|x|)>4/3$, it then implies from \eqref{A51WPI} that for all $|x|>N'$ we have
\begin{align}
    e^{f(|x|)}>C_{N,d} |x|^{(2-\theta)^{-1}\mu b^{-\frac{2-\theta}{\beta}}\log^{\frac{2-\theta}{\beta}-1}(|x|)+d}\log^{-\frac{d-\beta}{\beta}}(|x|). \label{A51WPILowBound}
\end{align}
Furthermore, \eqref{A52WPI} implies that for all $|x|>N'$,
\begin{align}
    e^{f(|x|)}>C_{N,d} |x|^{(1-\theta)^{-1}(2-\theta)^{-1}\mu b^{-\frac{2-\theta}{\beta}}\log^{\frac{2-\theta}{\beta}-1}(|x|)+C_N \log ^{\frac{1}{\beta}-1}(|x|)+1}\log^{-(d-\beta)} (|x|). \label{A52WPILowBound}
\end{align}
We now consider the different cases as before.
\begin{itemize}
    \item [(1)] When $\theta<2-\beta$, \eqref{A52WPILowBound} is stronger than \eqref{A51WPILowBound}. We have that for large $|x|$,
    \begin{align*}
        (1-\theta)^{-1}(2-\theta)^{-1}\mu b^{-\frac{2-\theta}{\beta}}\log^{\frac{2-\theta}{\beta}-1}(|x|)+C_N \log ^{\frac{1}{\beta}-1}(|x|)+1-(d+\vartheta)>0.
    \end{align*}
    Therefore, when $\theta<2-\beta$,  the conditions in \cite[Theorem 1.1-(3)]{wang2014simple} are satisfied with $$\omega(x)=\frac{C_{N,d}}{2^{d+\vartheta}}|x|^{(1-\theta)^{-1}(2-\theta)^{-1}\mu b^{-\frac{2-\theta}{\beta}}\log^{\frac{2-\theta}{\beta}-1}(|x|)+1-(d+\vartheta)}\log^{-\frac{d-\beta}{\beta}}(|x|),$$ with $\lim_{|x|\to \infty}\omega(x)=\infty$. Hence, $\pi\propto \exp(-f)$ satisfies the super Poincar\'e inequality.
    \item [(2)] When $\theta=2-\beta, \mu\beta^{-1}b^{-1}>\vartheta$, \eqref{A51WPILowBound} is stronger than \eqref{A52WPILowBound}. We have that for all $|x|>N'$,
    \begin{align*}
      (2-\theta)^{-1}\mu b^{-\frac{2-\theta}{\beta}}\log^{\frac{2-\theta}{\beta}-1}(|x|)+d-(d+\vartheta)>0.
    \end{align*}
    Therefore, when $\theta=2-\beta, \mu\beta^{-1}b^{-1}>\vartheta$, the conditions in \cite[Theorem 1.1-(3)]{wang2014simple} are satisfied with $$\omega(x)=\frac{C_{N,d}}{2^{d+\vartheta}}|x|^{b^{-\frac{2-\theta}{\beta}}\log^{\frac{2-\theta}{\beta}-1}(|x|)-\vartheta}\log^{-\frac{d-\beta}{\beta}}(|x|),$$ with $\lim_{|x|\to \infty}\omega(x)=\infty$. Hence, $\pi\propto \exp(-f)$ satisfies the super Poincar\'e inequality.
    \item [(3)] When $\theta=2-\beta, \mu\beta^{-1}b^{-1}\le \vartheta$, \eqref{A51WPILowBound} is stronger than \eqref{A52WPILowBound}. We have that for all $|x|$ large enough,
    \begin{align*}
       &(2-\theta)^{-1}\mu b^{-\frac{2-\theta}{\beta}}\log^{\frac{2-\theta}{\beta}-1}(|x|)+d-(d+\vartheta)\le 0.
    \end{align*}
   Neither Poincar\'e inequality nor super Poincar\'e inequality is guaranteed. According to \cite[Theorem 1.1-(2)]{wang2014simple}, $\pi\propto \exp(-f)$ satisfies the weak Poincar\'e inequality with $\alpha(r)$ in~\eqref{eq:alphar}. 
    \item [(4)] When $\theta>2-\beta$, \eqref{A51WPI} is stronger than \eqref{A52WPILowBound}. We have that for all $|x|$ large enough,
    \begin{align*}
         (2-\theta)^{-1}\mu b^{-\frac{2-\theta}{\beta}}\log^{\frac{2-\theta}{\beta}-1}(|x|)+d-(d+\vartheta)<0.
    \end{align*}
    Hencce, neither the Poincar\'e inequality nor the super Poincar\'e inequality is guaranteed. However, according to \cite[Theorem 1.1-(2)]{wang2014simple}, $\pi\propto \exp(-f)$ satisfies the weak Poincar\'e inequality with  $\alpha(r)$ in~\eqref{eq:alphar}.
\end{itemize}
\end{proof}

\subsection{Proofs for~\cref{sec:Examples}}\label{sec:secegproofs}
\begin{proof}[Proof of \cref{GinG0}:] First, it's easy to check that $g_{in}(0)=0$ and $g_{in}(b^{-\frac{1}{2}})=e$, which implies that $g\in \mc{C}([0,\infty))$.  Next note that we have
\begin{align*}
    \log \frac{g_{in}(r)}{r}=\log(b^{\frac{1}{2}})+br^2-\frac{10}{3}b^{\frac{3}{2}}r^3+\frac{15}{4}b^2r^4-\frac{6}{5}b^{\frac{5}{2}}r^5+\frac{47}{60}.
\end{align*}
Hence, we can then check that 
$$\lim_{r\to 0_+} \left|\frac{\frac{d}{dr}\log \frac{g_{in}(r)}{r}}{r}\right|<\infty~~\text{and}~~\lim_{r\to 0_+} \left|\frac{d}{dr^2} \log \frac{g_{in}(r)}{r}\right|<\infty.$$ 
Note that the first derivative of $g_{in}$ is given by 
\begin{align*}
    g_{in}'(r)=b^{\frac{1}{2}}\left(1+2br^2-10b^{\frac{3}{2}}r^3+15b^2r^4-6b^{\frac{5}{2}}r^5\right)\exp\left(br^2-\frac{10}{3}b^{\frac{3}{2}}r^3+\frac{15}{4}b^2r^4-\frac{6}{5}b^{\frac{5}{2}}r^5+\frac{47}{60}\right).
\end{align*}
Hence, we have that
\begin{align*}
    \lim_{r\to 0_+} |f'(g_{in}(r))g_{in}'(r)|=(d+\varepsilon)\lim_{r\to 0_+}\left|\frac{g_{in}'(r)}{1+g_{in}(r)^2}\right| \frac{g_{in}(r)}{r}<\infty.
\end{align*}
Similarly, as  $g_{in}'(b^{-\frac{1}{2}})=2b^{\frac{1}{2}}e$ and
\begin{align*}
    \log g_{in}'(r)=&\log (b^{\frac{1}{2}})+\log(1+2br^2-10b^{\frac{3}{2}}r^3+15b^2r^4-6b^{\frac{5}{2}}r^5)\\
    &+br^2-\frac{10}{3}b^{\frac{3}{2}}r^3+\frac{15}{4}b^2r^4-\frac{6}{5}b^{\frac{5}{2}}r^5+\frac{47}{60},
\end{align*}
we can also check that $$\lim_{r\to 0_+}\left|\frac{d^2}{dr^2} \log g_{in}'(r)\right|<\infty~~\text{and}~~\lim_{r\to 0_+} \left|\frac{\frac{d}{dr}\log g_{in}'(r)}{r}\right|<\infty.$$ Similarly, by taking additional higher order derivatives it is easy to check that $g_{in}^{''}(b^{-\frac{1}{2}})=6be$ and $g_{in}^{'''}(b^{-\frac{1}{2}})=20b^{\frac{3}{2}}e$. We omit the tedious but elementary calculations here.
\end{proof}

\begin{proof}[Proof of~\cref{lem:tdistribution}] 
 It's obvious that $f\in \mc{C}^2(\mb{R}^d)$ and it's isotropic. Note that
 \begin{align*}
     \frac{d}{d|x|} f(|x|)&=(d+\kappa)\frac{|x|}{1+|x|^2},\\
     \frac{d^2}{d|x|^2} f(|x|)&=(d+\kappa)\frac{1-|x|^2}{(1+|x|^2)^2}.
 \end{align*}
With $\psi(r)=e^{br^\beta}$ for all $r\ge b^{-\frac{1}{\beta}}$, based on \eqref{EO1} and \eqref{EO2}, we have that for all $|x|\gg b^{-\frac{1}{\beta}}$ and $k\in \mb{Z}^+$,
\begin{align*}
    f'(\psi(|x|))\psi'(|x|)|x|^{-1}-b \beta d |x|^{\beta-2}+(d-\beta)|x|^{-2}&=\kappa b \beta |x|^{\beta-2}+(d-\beta)|x|^{-2}+o(|x|^{-k}),
\end{align*}
and
\begin{align*}
    &f''(\psi(|x|))\psi'(|x|)^2+f'(\psi(|x|))\psi''(|x|)-b\beta(\beta-1)|x|^{\beta-2}-(d-\beta)|x|^{-2}\\
    &=\kappa b\beta(\beta-1)|x|^{\beta-2}-(d-\beta)|x|^{-2}+o(|x|^{-k}).
\end{align*}
Note that for all $|x|\ge b^{-\frac{1}{\beta}}$, we have
\begin{align*}
    &\kappa b \beta |x|^{\beta-2}\le \kappa \beta b^{\frac{2}{\beta}},\\
    &\kappa b\beta(\beta-1)|x|^{\beta-2}\le \kappa \beta(\beta-1) b^{\frac{2}{\beta}}\le \kappa \beta b^{\frac{2}{\beta}}.
\end{align*}
The last inequality holds since $\beta\in (1,2]$. Therefore $f$ satisfies \cref{ass:A2} with some $N_4>0$ and $L=2\kappa \beta b^{\frac{2}{\beta}}$. 

To check \cref{ass:A3}, notice that for all $|x|\gg b^{-\frac{1}{\beta}}$ and $k\in\mb{Z}^+$, we have
\begin{align*}
    f'(\psi(|x|))\psi'(|x|)|x|-b\beta d|x|^\beta +(d-\beta)=\kappa b\beta |x|^\beta+(d-\beta) +o(|x|^{-k})
\end{align*}
Therefore \cref{ass:A3} is satisfied with $A=\kappa b\beta$, $\alpha=\beta$ and some $B\ge 0$,$N_1>0$. 

Lastly, to check \cref{ass:A5}, similar to the calculation in checking \cref{ass:A2}, for all $|x|\gg b^{-\frac{1}{\beta}}$ and $k\in\mb{Z}^+$, we have
\begin{align*}
    f'(\psi(|x|))\psi'(|x|)|x|^{-1}-b\beta d|x|^{\beta-2}+(d-\beta)|x|^{-2}=\kappa b\beta |x|^{\beta-2}+(d-\beta)|x|^{-2}+o(|x|^{-k}),
\end{align*}
and 
\begin{align*}
    &f''(\psi(|x|))\psi'(|x|)^2+f'(\psi(|x|))\psi''(|x|)-b \beta(\beta-1)|x|^{\beta-2}-(d-\beta)|x|^{-2}\\
    &=\kappa b\beta(\beta-1)|x|^{\beta-2}-(d-\beta)|x|^{-2}+o(|x|^{-k}).
\end{align*}
Therefore \cref{ass:A5} is satisfied with arbitrary $\mu\in(0,\kappa b\beta(\beta-1))$, $\theta=2-\beta\ge 0$ and some $N_2>0$.
\end{proof}

\subsubsection{Order estimation of mixing time when $\beta=2$}\label{sec:orderest}
When $f(x)=\frac{d+\kappa}{2}\log (1+|x|^2)$, for all $|x|>b^{-\frac{1}{2}}$, the two eigenvalues of $\nabla^2 f_h(x)$ can be studied via \eqref{EO1} and \eqref{EO2}. We obtain 
\begin{align}
    \lambda_1&=2b\kappa+(d-2)|x|^{-2}-2b(d+\kappa)\frac{1}{1+e^{2b|x|^2}},  \label{multitE1}\\
    \lambda_2&=2b\kappa-(d-2)|x|^{-2}+2b(d+\kappa)\frac{(4b|x|^2-1)e^{2b|x|^2}+1}{(1+e^{2b|x|^2})^2}\label{multitE2}.
\end{align}
Therefore, for all $|x|>b^{-\frac{1}{2}}$: we can estimate $\lambda_1$:
\begin{align*}
    2b\kappa-2b\kappa\frac{1}{1+e^2}-2b<\lambda_1<2b\kappa+(d-2)b,
\end{align*}
which can be simplified as 
\begin{align}\label{L1estimation}
   2b(\frac{e^2}{1+e^2}\kappa-1)<\lambda_1<2b(\kappa+\frac{d}{2}-1),
\end{align}
for all $|x|>b^{-\frac{1}{2}}$. Similarly, we can obtain the following estimate on $\lambda_2$:
\begin{align*}
    2b\kappa-bd\frac{e^4-3e^2-1}{(1+e^2)^2}<\lambda_2<2b\kappa+2b+2b\kappa\frac{3e^2+1}{(1+e^2)^2}.
\end{align*}
The above estimation can be further simplified as 
\begin{align}\label{L2estimation}
  2b(\kappa-0.2d)<\lambda_2<2b(1.5\kappa+1).   
\end{align}
According to \eqref{L1estimation} and \eqref{L2estimation}, we instantly have the locally Lipschitz constant, denoted as $L_{h,loc}$, for $f_h$ in the region $\{|x|>b^{-\frac{1}{2}}\}$ being characterized as 
\begin{align*}
    L_{h,loc}=2b \max\left\{ \kappa+\frac{d}{2}-1,1.5\kappa+1 \right\}.
\end{align*}
Next, for $|x|\le b^{-\frac{1}{2}}$, we can check that for any fixed $d$, we have
\begin{align*}
    \lim_{|x|\to 0} |\lambda_i(|x|)| <\infty \quad i=1,2
\end{align*}
Therefore we can check that for any fixed $|x|\le b^{-\frac{1}{2}}$, we have $|\lambda_i(x)|=O(d)$ for $i=1,2$ when $d\gg 1$. Thus we can conclude the global Lipschitz constant of $f_h$, $L_h=O(d)$ for $d\gg 1$. 

On the other hand side, from \eqref{L1estimation}, we can see for all $\kappa>\frac{1+e^2}{e^2}$, $\lambda_1>b(\frac{e^2}{1+e^2}\kappa-1)$. While from \eqref{L2estimation}, the lower bound would be negative if $d\gg \kappa$. Therefore to ensure both eigenvalues are lower bounded by $b\kappa$, we need to restrict the region $\{|x|>b^{-\frac{1}{2}}\}$ to set of points with larger magnitudes. For all $|x|>(\frac{d}{b\kappa})^{\frac{1}{2}}$, we have when $d\ge \kappa$ and $d\ge 3$ that
\begin{align*}
    \lambda_1&>b\kappa\left(2-\frac{2}{1+e^{2d/\kappa}}-2/d\right)>b\kappa,  \\
    \lambda_2&>2b\kappa-(d-2)(b\kappa/d)=b(\kappa+\kappa/d)>b\kappa.
\end{align*}
To determine the LSI constant, we first construct a function $G_h$ such that $\nabla^2 G_h(x) \succeq b\kappa I_d$ for all $x \in \mathbb{R}^d$. Letting $\varpi:=\sqrt{({d}/{b\kappa})}$, the function $G_h$ is defined piecewisely as 
\begin{equation*}
    G_h = \left\{
    \begin{aligned}
    &f_h   &|x|>\varpi  \\
    &\frac{1}{3} A\left(|x|- \varpi\right)^3+ \frac{1}{2} f_h''\left(\varpi\right)\left(|x|-\varpi\right)^2 \\
    &+f_h'(\varpi)(|x|-\varpi)+f_h(\varpi) & |x|\le \varpi,
    \end{aligned}
    \right.
\end{equation*}
where
\begin{align*}
    f_h(\varpi)&=\frac{d}{2}\log(1+e^{-2d/\kappa})+\frac{\kappa}{2}\log (1+e^{2d/\kappa})+(d-2)\log (\varpi)-\log 2.
\end{align*}
Note that we also have
\begin{align*}
    f_h'(\varpi)&=2b\kappa\varpi+(d-2)\left(\frac{d}{b\kappa}\right)^{-\frac{1}{2}}-2bd(1+\kappa/d)\frac{\varpi}{1+e^{2d/\kappa}},\\
    f_h''(\varpi)&=b\kappa\left(1+\frac{2}{d}\right)+2bd(1+\kappa/d)\frac{\left(4\frac{d}{\kappa}-1\right)e^{2d/\kappa}+1}{(1+e^{2d/\kappa})^2},\\
    A&=-\frac{b\kappa}{d}\left(-2\varpi-4bd(1+\kappa/d)\varpi\frac{1-2\frac{d}{\kappa}e^{2d/\kappa}}{(1+e^{2d/\kappa})^2} \right)<0.
\end{align*}
With the above coefficients, we can check $G_h \in \mc{C}^2(\mb{R}^d)$ and $\nabla G_h(x)\succeq b\kappa I_d$ for all $x\in \mb{R}^d$. We now consider different cases.
\begin{itemize}[leftmargin=0.25in]
    \item [(1)] When $d\gg \kappa$ for all $k\in \mb{Z}$: 
\begin{align*}
    f_h(\varpi)&=d+\frac{1}{2}(d-1)\log d+O(1),\\
    f'(\varpi)&= 3d\left(\frac{d}{b\kappa} \right)^{-\frac{1}{2}}-2\left(\frac{d}{b\kappa}\right)^{-\frac{1}{2}}+o(d^{-k}), \\
    f''(\varpi)&=b\kappa+2\left(\frac{d}{b\kappa}\right)^{-1}+o(d^{-k}),\\
    A&=-2d\left(\frac{d}{b\kappa}\right)^{-\frac{3}{2}}+4\left(\frac{d}{b\kappa}\right)^{-\frac{3}{2}}+o(d^{-k}).
\end{align*}
Therefore the oscillation between $f_h$ and $G_h$ can be written as 
\begin{align*}
    Osc(f_h-G_h)&= \max_{0\le |x|\le\varpi} |f_h(x)-G_h(x)|.
\end{align*}
Since both $G_h$ and $f_h$ are monotone increasing with respect to $|x|$, we then have
\begin{align*}
    Osc(f_h-G_h)\le G_h(\varpi)+f_h(\varpi)=2d+(d-1)\log d+O(1)
\end{align*}
On the other hand,
\begin{align*}
    Osc(f_h-G_h)\ge G_h(0)-f_h(0)=\frac{1}{2}(d-1)\log d-\frac{5}{6} d+O(1)
\end{align*}
Hence, apply Holley-Strook lemma, we can calculate the LSI constant $\chlsi$ as
\begin{align*}
    \chlsi \le 2(b\kappa)^{-1} \exp( Osc(f_h-G_h) )\le C(b\kappa)^{-1} d^{d-1} \exp(2d).
\end{align*}
Furthermore, because of the lower bound on $Osc(f_h,G_h)$, the factor $d^{d-1}$ can be improved. Hence, according to~\cref{LSITULA}, to reach $\epsilon$-accuracy in KL-divergence, the mixing time $n$ satisfies:
\begin{align*}
    n\sim \Tilde{O}(L_h C_h d \epsilon^{-1})\le \Tilde{O}(\exp(2d) d^{d+1}\epsilon^{-1}).
\end{align*}
\item [(2)] When $d/\kappa=O(1)$, or equivalently when $d/\kappa\to C'$, we have
\begin{align*}
    f_h(\varpi)&=\frac{d}{2}[\log(1+e^{-2C'})+C'^{-1}\log (1+e^{2C'})+\log (\frac{C'}{b})]+O(1),\\
    &:=d C'_1+O(1),\\
    f'(\varpi)&=3d(\frac{C'}{b} )^{-\frac{1}{2}}-2(\frac{C'}{b})^{-\frac{1}{2}}+o(d^{-k}) \\
    &:= b^{\frac{1}{2}}d C'_2+O(1) \\
    f_h''(\varpi)&=bd C'^{-1}+2(\frac{C'}{b})^{-1}+o(d^{-k}) \\
     &:=bd C'_3 +O(1),\\
     A&=-2d(\frac{C'}{b})^{-\frac{3}{2}}+4(\frac{C'}{b})^{-\frac{3}{2}}+o(d^{-k}) \\
     &:=b^{\frac{3}{2}}C_4' d+O(1).
\end{align*}
Therefore for all $|x|\le (C'/b)^{\frac{1}{2}}$, we have
\begin{align*}
    G_h(x)&=d\left\{\frac{1}{3} b^{\frac{3}{2}}C_4'|x|^3 + b( \frac{1}{2} C'_3-C'^{\frac{1}{2}}C_4') |x|^2+b^{\frac{1}{2}} (C'_2-C'^{\frac{1}{2}}C'_3+C' C_4')|x|\right.\\
     &\ \left.+(C'_1-C'^{\frac{1}{2}}C'_2+\frac{1}{2} C'C'_3-\frac{1}{3}C'^{\frac{3}{2}}C'_4)\right\}+O(1)
\end{align*}
Similar to the previous argument, the oscillation can be upper bounded as
\begin{align*}
    Osc(f_h-G_h)&= G_h((\frac{C'}{b})^{\frac{1}{2}})+f_h((\frac{C'}{b})^{\frac{1}{2}}) \\
    &=C_h' d+O(1),
\end{align*}
where
\begin{align*}
    C_h'&=\frac{1}{3}C_4'C'^{\frac{3}{2}}+(\frac{1}{2}C_3'-C'^{\frac{1}{2}}C_4')C'+(C'_2-C'^{\frac{1}{2}}C'_3+C' C_4')C'^{\frac{1}{2}}\\
    &\ +(2C'_1-C'^{\frac{1}{2}}C'_2+\frac{1}{2} C'C'_3-\frac{1}{3}C'^{\frac{3}{2}}C'_4).
\end{align*}
Hence, applying Holley-Strook Theorem, the LSI constant can be bounded by
\begin{align*}
    \chlsi\le 2(b\kappa)^{-1}\exp(Osc(f_h-G_h)) \le C(bd/C')^{-1}(\exp(C_h'))^d.
\end{align*}
Hence, according to \cite{vempala2019rapid}, to reach $\epsilon$-accuracy in KL-divergence, the mixing time $n$ satisfies
\begin{align*}
    n\sim \Tilde{O}(L_h \chlsi d \epsilon^{-1})\le \Tilde{O}((\exp(C_h'))^d d^{-1}\epsilon^{-1}).
\end{align*}
\end{itemize}

\subsection*{Acknowledgements} YH was supported in part by NSF TRIPODS grant CCF-1934568. KB was supported in part by NSF grant DMS-2053918. MAE was supported by NSERC Grant [2019-06167], Connaught New Researcher Award, CIFAR AI Chairs program, and CIFAR AI Catalyst grant. We thank Sinho Chewi for providing several helpful remarks on this work. Parts of this work was done when the authors visited the Simons Institute for the Theory of Computing as a part of the ``Geometric Methods in Optimization and Sampling" program during Fall 2021.  
\bibliographystyle{alpha}
\bibliography{citation}

\newcommand{\etalchar}[1]{$^{#1}$}
\begin{thebibliography}{CCAY{\etalchar{+}}18}

\bibitem[BCG08]{bakry2008rate}
Dominique Bakry, Patrick Cattiaux, and Arnaud Guillin.
\newblock {Rate of convergence for ergodic continuous Markov processes:
  Lyapunov versus Poincar{\'e}}.
\newblock {\em Journal of Functional Analysis}, 254(3):727--759, 2008.

\bibitem[B{\'E}85]{bakry1985diffusions}
Dominique Bakry and Michel {\'E}mery.
\newblock Diffusions hypercontractives.
\newblock In {\em Seminaire de probabilit{\'e}s XIX 1983/84}, pages 177--206.
  Springer, 1985.

\bibitem[BGL14]{bakry2014analysis}
Dominique Bakry, Ivan Gentil, and Michel Ledoux.
\newblock {\em Analysis and geometry of Markov diffusion operators}, volume
  103.
\newblock Springer, 2014.

\bibitem[BRZ19]{bierkens2019ergodicity}
Joris Bierkens, Gareth~O Roberts, and Pierre-Andr{\'e} Zitt.
\newblock Ergodicity of the zigzag process.
\newblock {\em The Annals of Applied Probability}, 29(4):2266--2301, 2019.

\bibitem[CCAY{\etalchar{+}}18]{cheng2018sharp}
Xiang Cheng, Niladri~S Chatterji, Yasin Abbasi-Yadkori, Peter~L Bartlett, and
  Michael~I Jordan.
\newblock Sharp convergence rates for {L}angevin dynamics in the nonconvex
  setting.
\newblock {\em arXiv preprint arXiv:1805.01648}, 2018.

\bibitem[CD21]{cui2021deep}
Tiangang Cui and Sergey Dolgov.
\newblock Deep composition of tensor-trains using squared inverse {R}osenblatt
  transports.
\newblock {\em Foundations of Computational Mathematics}, pages 1--60, 2021.

\bibitem[CDWY20]{chen2020fast}
Yuansi Chen, Raaz Dwivedi, Martin~J Wainwright, and Bin Yu.
\newblock {Fast mixing of Metropolized Hamiltonian Monte Carlo: Benefits of
  multi-step gradients}.
\newblock {\em J. Mach. Learn. Res.}, 21:92--1, 2020.

\bibitem[CEL{\etalchar{+}}21]{chewi2021analysis}
Sinho Chewi, Murat~A Erdogdu, Mufan~(Bill) Li, Ruoqi Shen, and Matthew Zhang.
\newblock {Analysis of Langevin Monte Carlo from Poincar\'{e} to Log-Sobolev}.
\newblock {\em arXiv preprint arXiv:2112.12662}, 2021.

\bibitem[CGGR10]{cattiaux2010functional}
Patrick Cattiaux, Nathael Gozlan, Arnaud Guillin, and Cyril Roberto.
\newblock Functional inequalities for heavy tailed distributions and
  application to isoperimetry.
\newblock {\em Electronic Journal of Probability}, 15:346--385, 2010.

\bibitem[CLA{\etalchar{+}}21]{chewi2021optimal}
Sinho Chewi, Chen Lu, Kwangjun Ahn, Xiang Cheng, Thibaut Le~Gouic, and Philippe
  Rigollet.
\newblock {Optimal dimension dependence of the Metropolis-Adjusted Langevin
  Algorithm}.
\newblock In {\em Conference on Learning Theory}, pages 1260--1300. PMLR, 2021.

\bibitem[CW97]{chen1997estimates}
Mu-Fa Chen and Feng-Yu Wang.
\newblock Estimates of logarithmic {S}obolev constant: An improvement of
  {Bakry--Emery} criterion.
\newblock {\em Journal of functional analysis}, 144(2):287--300, 1997.

\bibitem[Dal17]{dalalyan2017theoretical}
Arnak~S Dalalyan.
\newblock Theoretical guarantees for approximate sampling from smooth and
  log-concave densities.
\newblock {\em Journal of the Royal Statistical Society: Series B},
  79(3):651--676, 2017.

\bibitem[DBCD19]{deligiannidis2019exponential}
George Deligiannidis, Alexandre Bouchard-C{\^o}t{\'e}, and Arnaud Doucet.
\newblock Exponential ergodicity of the bouncy particle sampler.
\newblock {\em The Annals of Statistics}, 47(3):1268--1287, 2019.

\bibitem[DCWY19]{dwivedi2019log}
Raaz Dwivedi, Yuansi Chen, Martin~J Wainwright, and Bin Yu.
\newblock Log-concave sampling: {Metropolis-Hastings} algorithms are fast.
\newblock {\em Journal of Machine Learning Research}, 20:1--42, 2019.

\bibitem[DGM20]{durmus2020geometric}
Alain Durmus, Arnaud Guillin, and Pierre Monmarch{\'e}.
\newblock Geometric ergodicity of the bouncy particle sampler.
\newblock {\em The Annals of Applied Probability}, 30(5):2069--2098, 2020.

\bibitem[DK19]{dalalyan2019user}
Arnak~S Dalalyan and Avetik Karagulyan.
\newblock User-friendly guarantees for the {Langevin Monte Carlo} with
  inaccurate gradient.
\newblock {\em Stochastic Processes and their Applications},
  129(12):5278--5311, 2019.

\bibitem[DKRD19]{dalalyan2019bounding}
Arnak~S Dalalyan, Avetik Karagulyan, and Lionel Riou-Durand.
\newblock Bounding the error of discretized {L}angevin algorithms for
  non-strongly log-concave targets.
\newblock {\em arXiv preprint arXiv:1906.08530}, 2019.

\bibitem[DM17]{durmus2017nonasymptotic}
Alain Durmus and Eric Moulines.
\newblock Nonasymptotic convergence analysis for the {U}nadjusted {L}angevin
  {A}lgorithm.
\newblock {\em The Annals of Applied Probability}, 27(3):1551--1587, 2017.

\bibitem[DMM19]{durmus2019analysis}
Alain Durmus, Szymon Majewski, and B{\l}a{\.z}ej Miasojedow.
\newblock Analysis of {Langevin Monte Carlo} via convex optimization.
\newblock {\em The Journal of Machine Learning Research}, 20(1):2666--2711,
  2019.

\bibitem[EH21]{erdogdu2020convergence}
Murat~A Erdogdu and Rasa Hosseinzadeh.
\newblock {On the convergence of Langevin Monte Carlo: The interplay between
  tail growth and smoothness}.
\newblock In {\em Conference on Learning Theory}, pages 1776--1822. PMLR, 2021.

\bibitem[EMS18]{erdogdu2018global}
Murat~A Erdogdu, Lester Mackey, and Ohad Shamir.
\newblock Global non-convex optimization with discretized diffusions.
\newblock In {\em Proceedings of the 32nd International Conference on Neural
  Information Processing Systems}, pages 9694--9703, 2018.

\bibitem[Erb14]{erbar2014gradient}
Matthias Erbar.
\newblock Gradient flows of the entropy for jump processes.
\newblock In {\em Annales de l'IHP Probabilit{\'e}s et Statistiques},
  volume~50, pages 920--945, 2014.

\bibitem[GB09]{genz2009computation}
Alan Genz and Frank Bretz.
\newblock {\em Computation of multivariate normal and t-probabilities}, volume
  195.
\newblock Springer Science \& Business Media, 2009.

\bibitem[GBH04]{genz2004approximations}
Alan Genz, Frank Bretz, and Yosef Hochberg.
\newblock Approximations to multivariate $ t $ integrals with application to
  multiple comparison procedures.
\newblock In {\em Recent Developments in Multiple Comparison Procedures}, pages
  24--32. Institute of Mathematical Statistics, 2004.

\bibitem[HBE20]{he2020ergodicity}
Ye~He, Krishnakumar Balasubramanian, and Murat~A Erdogdu.
\newblock {On the Ergodicity, Bias and Asymptotic Normality of Randomized
  Midpoint Sampling Method}.
\newblock {\em Advances in Neural Information Processing Systems}, 33, 2020.

\bibitem[HMW21]{huang2021approximation}
Lu-Jing Huang, Mateusz~B Majka, and Jian Wang.
\newblock Approximation of heavy-tailed distributions via stable-driven {SDE}s.
\newblock {\em Bernoulli}, 27(3):2040--2068, 2021.

\bibitem[HS87]{holley1986logarithmic}
Richard Holley and Daniel Stroock.
\newblock {Logarithmic Sobolev Inequalities and stochastic Ising models}.
\newblock {\em Journal of Statistical Physics}, 46(5-6):1159--1194, 1987.

\bibitem[JG12]{johnson2012variable}
Leif~T Johnson and Charles~J Geyer.
\newblock Variable transformation to obtain geometric ergodicity in the
  {R}andom-{W}alk {M}etropolis algorithm.
\newblock {\em The Annals of Statistics}, 40(6):3050--3076, 2012.

\bibitem[JR07]{jarner2007convergence}
S{\o}ren~F Jarner and Gareth~O Roberts.
\newblock Convergence of heavy-tailed {Monte Carlo Markov Chain} algorithms.
\newblock {\em Scandinavian Journal of Statistics}, 34(4):781--815, 2007.

\bibitem[Kam18]{kamatani2018efficient}
Kengo Kamatani.
\newblock Efficient strategy for the {Markov chain Monte Carlo} in
  high-dimension with heavy-tailed target probability distribution.
\newblock {\em Bernoulli}, 24(4B):3711--3750, 2018.

\bibitem[KMR21]{koehler2021representational}
Frederic Koehler, Viraj Mehta, and Andrej Risteski.
\newblock Representational aspects of depth and conditioning in normalizing
  flows.
\newblock In {\em International Conference on Machine Learning}, pages
  5628--5636. PMLR, 2021.

\bibitem[KN04]{kotz2004multivariate}
Samuel Kotz and Saralees Nadarajah.
\newblock {\em Multivariate t-distributions and their applications}.
\newblock Cambridge University Press, 2004.

\bibitem[Kwa17]{kwasnicki2017ten}
Mateusz Kwa{\'s}nicki.
\newblock Ten equivalent definitions of the fractional laplace operator.
\newblock {\em Fractional Calculus and Applied Analysis}, 20(1):7--51, 2017.

\bibitem[Leh21]{lehec2021langevin}
Joseph Lehec.
\newblock The {Langevin Monte Carlo} algorithm in the non-smooth log-concave
  case.
\newblock {\em arXiv preprint arXiv:2101.10695}, 2021.

\bibitem[LO00]{latala2000between}
Rafa{\l} Lata{\l}a and Krzysztof Oleszkiewicz.
\newblock {Between Sobolev and Poincar{\'e}}.
\newblock In {\em Geometric aspects of functional analysis}, pages 147--168.
  Springer, 2000.

\bibitem[LST20]{lee2020logsmooth}
Yin~Tat Lee, Ruoqi Shen, and Kevin Tian.
\newblock Logsmooth gradient concentration and tighter run-times for
  {Metropolized Hamiltonian Monte Carlo}.
\newblock In {\em Conference on Learning Theory}, pages 2565--2597, 2020.

\bibitem[LWME19]{li2019stochastic}
Xuechen Li, Yi~Wu, Lester Mackey, and Murat~A Erdogdu.
\newblock Stochastic {Runge-Kutta} accelerates {L}angevin {M}onte {C}arlo and
  beyond.
\newblock {\em Advances in Neural Information Processing Systems},
  32:7748--7760, 2019.

\bibitem[MCJ{\etalchar{+}}19]{ma2019sampling}
Yi-An Ma, Yuansi Chen, Chi Jin, Nicolas Flammarion, and Michael~I Jordan.
\newblock Sampling can be faster than optimization.
\newblock {\em Proceedings of the National Academy of Sciences},
  116(42):20881--20885, 2019.

\bibitem[MMS20]{majka2020nonasymptotic}
Mateusz~B Majka, Aleksandar Mijatovi{\'c}, and {\L}ukasz Szpruch.
\newblock Nonasymptotic bounds for sampling algorithms without log-concavity.
\newblock {\em The Annals of Applied Probability}, 30(4):1534--1581, 2020.

\bibitem[Ngu21]{nguyen2021unadjustedb}
Dao Nguyen.
\newblock {U}nadjusted {L}angevin {A}lgorithm for sampling a mixture of weakly
  smooth potentials.
\newblock {\em arXiv preprint arXiv:2112.09311}, 2021.

\bibitem[N{\c{S}}R19]{nguyen2019non}
Than~Huy Nguyen, Umut {\c{S}}im{\c{s}}ekli, and Ga{\"e}l Richard.
\newblock Non-asymptotic analysis of {Fractional Langevin Monte Carlo} for
  non-convex optimization.
\newblock In {\em International Conference on Machine Learning}, pages
  4810--4819, 2019.

\bibitem[PMM16]{parno2016multiscale}
Matthew Parno, Tarek Moselhy, and Youssef Marzouk.
\newblock A multiscale strategy for {B}ayesian inference using transport maps.
\newblock {\em SIAM/ASA Journal on Uncertainty Quantification},
  4(1):1160--1190, 2016.

\bibitem[PNR{\etalchar{+}}21]{papamakarios2021normalizing}
George Papamakarios, Eric Nalisnick, Danilo~Jimenez Rezende, Shakir Mohamed,
  and Balaji Lakshminarayanan.
\newblock Normalizing flows for probabilistic modeling and inference.
\newblock {\em Journal of Machine Learning Research}, 22(57):1--64, 2021.

\bibitem[RT96]{roberts1996exponential}
Gareth~O Roberts and Richard~L Tweedie.
\newblock Exponential convergence of {L}angevin distributions and their
  discrete approximations.
\newblock {\em Bernoulli}, pages 341--363, 1996.

\bibitem[RW01]{rockner2001weak}
Michael R{\"o}ckner and Feng-Yu Wang.
\newblock {Weak Poincar{\'e} inequalities and $L_2$ convergence rates of Markov
  semigroups}.
\newblock {\em Journal of Functional Analysis}, 185(2):564--603, 2001.

\bibitem[RW03]{rockner2003harnack}
Michael R{\"o}ckner and Feng-Yu Wang.
\newblock Harnack and functional inequalities for generalized {M}ehler
  semigroups.
\newblock {\em Journal of Functional Analysis}, 203(1):237--261, 2003.

\bibitem[SBM18]{spantini2018inference}
Alessio Spantini, Daniele Bigoni, and Youssef Marzouk.
\newblock Inference via low-dimensional couplings.
\newblock {\em The Journal of Machine Learning Research}, 19(1):2639--2709,
  2018.

\bibitem[{\c{S}}im17]{csimcsekli2017fractional}
Umut {\c{S}}im{\c{s}}ekli.
\newblock {Fractional Langevin Monte Carlo: Exploring L{\'e}vy driven
  stochastic differential equations for Markov Chain Monte Carlo}.
\newblock In {\em International Conference on Machine Learning}, pages
  3200--3209, 2017.

\bibitem[SL19]{shen2019randomized}
Ruoqi Shen and Yin~Tat Lee.
\newblock The {Randomized Midpoint Method} for log-concave sampling.
\newblock In {\em Proceedings of the 33rd International Conference on Neural
  Information Processing Systems}, pages 2100--2111, 2019.

\bibitem[{\c{S}}ZTG20]{simsekli2020fractional}
Umut {\c{S}}im{\c{s}}ekli, Lingjiong Zhu, Yee~Whye Teh, and Mert Gurbuzbalaban.
\newblock Fractional underdamped {L}angevin dynamics: Retargeting {SGD} with
  momentum under heavy-tailed gradient noise.
\newblock In {\em International Conference on Machine Learning}, pages
  8970--8980, 2020.

\bibitem[TIT{\etalchar{+}}20]{teshima2020coupling}
Takeshi Teshima, Isao Ishikawa, Koichi Tojo, Kenta Oono, Masahiro Ikeda, and
  Masashi Sugiyama.
\newblock Coupling-based invertible neural networks are universal
  diffeomorphism approximators.
\newblock {\em Advances in Neural Information Processing Systems (NeurIPS)},
  2020.

\bibitem[TV00]{toscani2000trend}
Guiseppe Toscani and C{\'e}dric Villani.
\newblock On the trend to equilibrium for some dissipative systems with slowly
  increasing a priori bounds.
\newblock {\em Journal of Statistical Physics}, 98(5):1279--1309, 2000.

\bibitem[VW19]{vempala2019rapid}
Santosh Vempala and Andre Wibisono.
\newblock Rapid convergence of the {U}nadjusted {L}angevin {A}lgorithm:
  {I}soperimetry suffices.
\newblock In {\em Advances in Neural Information Processing Systems}, pages
  8094--8106, 2019.

\bibitem[Wan06]{wang2006functional}
Feng-Yu Wang.
\newblock {\em Functional inequalities Markov semigroups and spectral theory}.
\newblock Elsevier, 2006.

\bibitem[Wan14]{wang2014simple}
Jian Wang.
\newblock A simple approach to functional inequalities for non-local
  {D}irichlet forms.
\newblock {\em ESAIM: Probability and Statistics}, 18:503--513, 2014.

\bibitem[WSC21]{wu2021minimax}
Keru Wu, Scott Schmidler, and Yuansi Chen.
\newblock Minimax mixing time of the {Metropolis-Adjusted Langevin Algorithm}
  for log-concave sampling.
\newblock {\em arXiv preprint arXiv:2109.13055}, 2021.

\bibitem[WW15]{wang2015functional}
Feng-Yu Wang and Jian Wang.
\newblock Functional inequalities for stable-like {D}irichlet forms.
\newblock {\em Journal of Theoretical Probability}, 28(2):423--448, 2015.

\end{thebibliography}
\appendix
\section{A Summary of Constants} 
For the sake of convenience, we provide a list of constants in Table~\ref{tab:table1}.
\begin{table}[p]
\centering
\begin{tabular}{ |c|c|c| } 
 \hline
 Constant & Description & Equation \\ \hline\hline
  $\epsilon$ & Accuracy parameter & NA\\ \hline 
 $\gamma$ & Step-size parameter & \eqref{TULan}\\ \hline \hline
 $\cpoincare$& Poincar\'{e} constant & \eqref{eq:PI} \\ \hline
 $\clsi $& LSI constant & \eqref{eq:LSI} \\ \hline
 $\cmlsi$, $\delta$ & m-LSI related constants & \eqref{eq:mLSI} \\ \hline
 \hline
  $\chpoincare$& Poincar\'{e} constant after Transformation &  NA\\ \hline
 $\chlsi $& LSI constant after Transformation & NA \\ \hline
 $\chmlsi$ & m-LSI constant after Transformation & NA \\ \hline 
 \hline
 $r, b, \beta$ & Parameters related to transformation map & \eqref{G1} \\ \hline\hline
 $A, B, N_1, \alpha$ & Parameters related to dissipativity & \cref{ass:A3} \\ \hline
 $\mu, N_2, \theta$ & Parameters related to degenerate convexity & \cref{ass:A5} \\ \hline
 $ N_3, \rho$ & Parameters related to convexity &  \cref{ass:A1} \\ \hline
 $ N_4, L$ & Parameters related to Lipschitz-gradients  &  \cref{ass:A2} \\ \hline
 $ N_5, m, \alpha_1,\ctail^* $ & Parameters related to tail condition &  \cref{ass:tail assumption on original potential} \\ \hline \hline
 $\alpha_h, A_h, B_h$ & Dissipativity parameters after transformation & \cref{alphadis} \\ \hline
 $\xi_h,\mu_h, \theta_h$ & Degenerate Convexity at infinity after transformation & \cref{ass:Degenerate convexity at infinity} \\ \hline
 $\rho_h$ & Strong-convexity parameter after transformation & NA \\ \hline
 $L_h$ & Lipschitz-gradient parameters after transformation & NA \\ \hline
 $m_h, \alpha_{h,1}, \chtail$ & Tail condition parameters after transformation & \cref{ass:tail assumption for transformed density} \\ \hline \hline
 $\kappa$ & Degrees-of-freedom of $t$ distribution & NA \\ \hline  $\vartheta$ & Parameter related to super and weak Poincar\'e inequalities & NA \\ \hline \hline
 \end{tabular}
 \caption{A list of all the constants used and their  description.}
 \label{tab:table1}
 \end{table}
\end{document}